\documentclass{article}

\usepackage[english]{babel}

\usepackage[top=2cm,bottom=2cm,left=3cm,right=3cm,marginparwidth=1.75cm]{geometry}

\usepackage{lipsum}
\usepackage{graphicx}
\usepackage{mathtools}
\usepackage{extarrows}
\usepackage{fullpage}
\usepackage{amsmath,amsthm,amsfonts,amscd}
\usepackage{amssymb}
\usepackage{lastpage}
\usepackage{enumitem}
\usepackage{fancyhdr}
\usepackage{mathrsfs}
\usepackage{xcolor}
\usepackage{listings}
\usepackage{hyperref}
\usepackage{bbm}
\usepackage{siunitx}
\PassOptionsToPackage{hyphens}{url}\usepackage{hyperref}
\usepackage[capitalize]{cleveref}
\usepackage[utf8]{inputenc}
\usepackage[right]{lineno}
\usepackage{csquotes}
\usepackage{booktabs}
\usepackage{longtable}
\usepackage{adjustbox}
\usepackage{array}
\usepackage{url}
\usepackage{titlesec}
\usepackage{authblk}
\usepackage{xcolor} 

\newtheorem{Proposition}{Proposition}
\newtheorem{Definition}{Definition}

\newtheorem{Theorem}{Theorem}
\newtheorem{Lemma}{Lemma}

\newtheorem{Corollary}{Corollary}

\usepackage[utf8]{inputenc}

\titleformat{\subsection}
  {\mdseries\itshape\large} 
  {\thesubsection}{1em}{} 

\usepackage[english]{babel}
\usepackage[numbers]{natbib}

\crefformat{figure}{#2Figure~#1#3}
\Crefformat{figure}{#2Figure~#1#3}
\crefformat{table}{#2Table~#1#3}
\Crefformat{table}{#2Table~#1#3}
\crefformat{section}{#2Section~#1#3}
\Crefformat{section}{#2Section~#1#3}

\author[1,2,3]{Daniel Solano}
\author[2]{Jerome Darbon}
\author[1,3]{Laurent Younes}

\affil[1]{Johns Hopkins University, Department of Applied Mathematics and Statistics, Baltimore, USA}
\affil[2]{Brown University, Division of Applied Mathematics, Providence, USA}
\affil[3]{Johns Hopkins University, Center for Imaging Science, Baltimore, USA}

\title{Shape Alignment via Allen-Cahn Nonlinear-Convection}

\begin{document}
\maketitle

\begin{abstract}
This paper demonstrates the impact of a phase field method on shape registration  to align shapes of possibly different topology. It yields new insights into the building of discrepancy measures between shapes regardless of topology, which would have applications in fields of image data analysis such as computational anatomy. A soft end-point optimal control problem is introduced whose minimum measures the minimal control norm required to align an initial shape to a final shape, up to a small error term.  The initial data is spatially integrable, the paths in control spaces are integrable and the evolution equation is a generalized convective Allen-Cahn. Binary images are used to represent shapes for the initial data. Inspired by level-set methods and large diffeomorphic deformation metric mapping, the controls spaces are integrable scalar functions to serve as a normal velocity and smooth reproducing kernel Hilbert spaces to serve as velocity vector fields. The existence of mild solutions to the evolution equation is proved, the minimums of the time discretized optimal control problem are characterized, and numerical simulations of minimums to the fully discretized optimal control problem are displayed. The numerical implementation enforces the maximum-bounded principle, although it is not proved for these mild solutions. This research offers a novel discrepancy measure that provides valuable ways to analyze diverse image data sets. Future work involves proving the existence of minimums, existence and uniqueness of strong solutions and the maximum bounded principle. 
\end{abstract}

\bibliographystyle{plainnat}
\section{Introduction}
\label{sec:introduction}

In his seminal treatise ``On Growth and Forms'' \cite{thompson1917growth}, D'Arcy Thompson introduced the fundamental principle of analyzing the similarities between two shapes  through the construction of transformations mapping one of them into the other. This principle was later formalized in Grenander's pattern theory \cite{grenander1991shape,grenander1993general}, the theory developing into practical tools with the introduction of diffeomorphic registration \cite{ShapesDiff} and, among others, the large deformation diffeomorphic mapping algorithm \cite{LDDMM3,LDDMM1,trouve1996habilitation,trouve1998diffeomorphism,dupuis1998variational,avants2006geodesic}, or LDDMM. This approach, along with many others targeting diffeomorphic registration \cite{ashburner1999nonlinear,friston1995spatial,DrRu04,vercauteren2009diffeomorphic,ashburner2009computational}, has been used for applications that extract information from shape changes, including, in particular, medical studies based on computational anatomy \cite{grenander1998computational}. 

These methods respect the basic paradigm of D'Arcy Thompson theory of transformations in that they evaluate, for comparison purposes, the ``simplest'' diffeomorphic transformation mapping one shape onto another. Restricting to such transformations has several advantages, one of the foremost being the structure of the diffeomorphism groups and the existence of well-studied Riemannian metrics \cite{arnold1978mathematical,michor2007some}, with transformations conveniently represented as flows of ordinary differential equations \cite{christensen1996deformable, trouve1998diffeomorphism, ShapesDiff}. In addition, diffeomorphisms are easily interpretable, as they can register target shapes to a fixed ``template'' allowing for the transport of existing information (such as region labeling) from the latter to the former.  Moreover, morphometric measures based on the deformation Jacobian are easily computable and can be used in further analysis \cite{ashburner2009computational}. 

Shape spaces represented as sets of diffeomorphic transformations of a template can be build to include a wild variety of objects, albeit all within the same topological class. Eliminating this topological constraint is challenging, however, as it may require to relinquish many of the properties that were just listed. Non-diffeomorphic comparison between images have been introduced in the context of metamorphosis \cite{miller2001group,trouve2005metamorphoses,holm2009euler,richardson2013computing,Casey},  which estimates transformations between images by optimizing displacements both in space and image intensity. There is no canonical modification of this approach to comparing shapes, defined as curves or surfaces, however, in that  running metamorphosis between two images with given zero level, would depend on the choice made for this representation, and not on the level sets themselves. This limitation is partially addressed in \citep{hsieh2022weight}, in which shapes are embedded in a much larger space of varifolds, which can be described, in codimension one, as products of measures on $\mathbb R^d$ (the varifold weight) and transition probabilities from $\mathbb R^d$ to the unit sphere. The metamorphosis mechanism is then applied to the weight component of the varifold. Since not all varifolds correspond to geometric curves or surfaces, this representation does not define, however, shape trajectories. Other work addressing the partial matching problem also include \citep{antonsanti2021partial,sukurdeep2022new}.

Of course, the ability to numerically generate topological changes by tracking the level sets of a function satisfying partial differential equations has a long history \cite{OSHER,osher2001level,osher2004level,sethian1996theory}, but such processes, which have found a wide range of applications in shape regularization, using, in particular, mean-curvature motion \cite{gage1986heat,grayson1987heat,evans1991motion}, and in shape segmentation \cite{chan2001active,vese2002multiphase,chan2006algorithms,malladi2002shape},  have had limited impact in shape analysis, mainly because these equations simplify their solutions (and the associated level sets) in a non-reversible way, which limits their usefulness in transporting spatial information.

In this paper, we provide a new strategy for shape analysis, in which a template shape is controlled by a combination of a diffeomorphic and normal flows while being submitted to a small amount of mean-curvature motion, in order to align with a target shape that is only subject to the mean-curvature motion.
Shapes being described as two- or three-dimensional sets with finite Lebesgue measure, their evolution is described using a  
nonlinear parabolic partial differential evolution equation of their characteristic functions 
We prove the global existence and uniqueness of this parabolic equation---specifically a generalized convective Allen Cahn equation. We set up a time discretization of an optimal control problem with running and terminal cost and prove a Pontryagin’s maximum principle for this problem. And finally, we show some numerical results in which we transform the characteristic function of an initial shape into the mollified characteristic function of a target shape with possibly different topology. Our critical contribution is in designing a meaningful discrepancy measure between shapes that is based on the minimal energy spent transforming one shape into another shape, regardless of whether they share the same topology. 

\Cref{sec:definitions} provides the statement of our main theorem, preceded by the necessary definitions and background results, while \cref{sec:opt.control} sets up the optimal control problem implemented in our experiments. The main theorem is proved in \cref{sec:main_section}. \Cref{sec:opt.cont} describes the discrete-time implementation of the method as an optimal control problems and make explicit the discrete Pontryagin maximum principle providing necessary conditions satisfied by solutions. Finally, \cref{sec:experiments} provides experimental illustrations of the method.

\section{Main theorem}
\subsection{Summary of notation}
If $B$ is a Banach space and $r\geq 1$, $L^r(\Omega, B)$ is the space of $r$-integrable functions (in the Bochner sense) from a measured space $\Omega$ to $B$ with the usual $L^r$ norm that we will denote $\|\ \|_{L^r}$. The notation is reduced to $L^r(\Omega)$ when $\Omega = \mathbb{R}$. We will also only write $L^r$ when $\Omega$ and $B$ are clear from the context. Given $\ell\geq 1$, we let $\ell^* = \ell/(\ell-1)$ such that $1/\ell + 1/\ell^* = 1$. 

The spaces $W^{k,r}(\mathbb R^d, \mathbb R^q)$ are the usual Sobolev spaces of functions $h:\mathbb R^d\to\mathbb R^q$ with all partial derivatives in $L^r$, and norm
\[
\|h\|_{W^{k,r}} = \left(\sum_{\alpha, |\alpha|\leq k} \|\partial_\alpha h\|_r^r\right)^{1/r}
\]
where the sum is over all multi-indexes of order $k$ or less, and $\partial_\alpha$ refers to partial derivatives with respect to coordinates indexed by the indexes in $\alpha$.  We will only write $W^{k,r}$ when $d$ and $q$ are clear from context, and we note that $W^{0,r} = L^r$.

The Banach space of bounded $k$-times continuously differentiable functions from $\mathbb{R}^d$ to $\mathbb{R}^q$ is denoted $C^k(\mathbb R^d, \mathbb R^q)$, with norm
\[
\|h\|_{C^p} = \max_{\alpha: |\alpha|\leq k} \sup_{x\in \mathbb R^d} |h(x)|
\]
(where $|y|$ denotes the Euclidean norm in finite dimensional spaces). The space $C^k_0(\mathbb R^d, \mathbb R^q)$ of functions that tend to 0 at infinity is the completion in $C^k(\mathbb R^d, \mathbb R^q)$ of compactly supported functions. Again, we will only write $C^k$ and $C^k_0$ when domains and codomains are clear from context.

We let $V$ be a Hilbert space of vector fields that is continuously embedded in $C^1(\mathbb{R}^d,\mathbb{R}^d)$. 
Within  large diffeomorphic deformation metric mapping, LDDMM, metrics (distances) are built between embeddings $S_0,S_1\subset \mathbb{R}^d$ of a chosen $C^1$ oriented manifold $M$. For these $S_0,S_1\subset \mathbb{R}^d$, one determines a  $v\in L^2([0,1],V)$ with minimal norm such that its flow, which satisfies\footnote{We will denote time-dependent vector fields either as $t\mapsto v(t)$ where $v(t)$ is a vector field or $t\mapsto v(t, \cdot)$, which is unambiguous within the spaces we will consider.}
$$\varphi_{0t}^v(x) = x + \int_0^t v(s,\varphi_{0s}^v(x)) ds, $$
maps $S_0$ to $S_1$ at time one, i.e., $\varphi^{v}_{01}(S_0) = S_1$. The minimum norm, written  $d_{V}(S_0,S_1)\in [0, +\infty]$, is a metric on the space of embeddings of $M$, written ${\rm Emb}(M,\mathbb{R}^d)$; see \cite{DiffonImm,LDDMM0,LDDMM1,LDDMM2,LDDMM3,LDDMM4} for more details on LDDMM and \cite{Overview,BAUER2019613} for an overview of shape analysis methods. In practice, one relaxes the boundary conditions into a soft end-point problem: minimize
$$\int_0^1 \|v(t)\|^2_{V} dt + U(\varphi^{v}_{01}), \quad v\in L^2([0,1],V),$$ where $U:C^{1}(\mathbb{R}^d,\mathbb{R}^d)\to \mathbb{R}$ is a continuous function bounded from below, which penalizes the lack of alignment between $S_1$ and $\varphi^{v}_{01}(S_0) = S_1$. Various choices are possible; check Chapters 10 and 11 of \cite{ShapesDiff} for more details.

\subsection{Useful properties of convolutions}
The convolution between two functions $f$ and $g$ is denoted $f\star g$. It is a bilinear mapping from $L^1\times L^r$ to $L^r$, which is bounded with \cite{brezis2011functional}
\begin{equation}
\label{eq:conv.basis}
\|f\star g\|_{L^r} \leq \|f\|_{L^1} \|g\|_{L^r}.
\end{equation}

We denote by $\eta_\tau$ the Gaussian function $x\mapsto \dfrac{\exp(-|x|^2/2\tau^2)}{(2\pi)^{d/2} \tau^d}$.  The convolution by $\eta_\tau$ is a bounded linear mapping between various pairs of states, as stated in the following lemma.
\begin{Lemma}(Chapter 15, 1.15, Third Equation,\cite{taylor})
\label{taylorlemma}
Let $q\geq \tilde q\in [1,\infty)$, $k\geq l\in \mathbb{R}$, and $\tau >0$. Then $f \mapsto \eta_\tau \star f$ is a bounded linear mapping from  $W^{l,\tilde q}(\mathbb{R}^d)$ to  $W^{k,q}(\mathbb{R}^d)$ with, for some $C>0$,
$$\|\eta_{\tau} \star f\|_{ W^{k,q}}\leq  C \tau^{-2\gamma}\|f\|_{ W^{l,\tilde q}},$$
$f\in W^{l,\tilde q}(\mathbb{R}^d)$ and where 
$$\gamma  = \dfrac{d}{2}\left(\dfrac{1}{\tilde q}-\dfrac{1}{q}\right)+\dfrac{1}{2}(k-l).$$
\end{Lemma}
\begin{proof}
We give a rapid proof of this result for completeness.
Let $a$ be a $d$-dimensional multi-index with $|a| \leq k$ and choose $b$ with $b \leq a$ and $|b| = \min(|a|, l)$. Let $c = a-b$, so that $|c| \leq k-l$. We have $\partial_a (\eta_\tau \star f) = (\partial_{c} \eta_\tau) \star (\partial_b f)$, and we let below $\bar\eta_\tau = \partial_c \eta_\tau$ and $g = \partial_b f$. We have $\bar \eta_\tau \star g = (\bar \eta_\tau^{1/\tilde q^*} \bar \eta_\tau^{1/\tilde q}) \star g$ yielding 
\[
|\bar \eta_\tau \star g| \leq \|\bar \eta_\tau\|_{L^1}^{1/\tilde q^*} |\bar \eta_\tau \star g^{\tilde q}|^{1/\tilde q}.
\]
This shows that
\[
\|\bar \eta_\tau \star g\|_{L^q} \leq \|\bar \eta_\tau\|_{L^1}^{1/\tilde q^*} \|\bar \eta_\tau \star g^{\tilde q}\|_{L^{q/\tilde q}}^{1/\tilde q} \leq \|\bar \eta_\tau\|_{L^1}^{1/\tilde q^*} \|g^{\tilde q}\|^{1/\tilde q}_{L^1} \|\bar \eta_\tau\|_{L^{q/\tilde q}}^{1/\tilde q} = \|\bar \eta_\tau\|_{L^1}^{1/\tilde q^*} \|\bar \eta_\tau\|_{L^{q/\tilde q}}^{1/\tilde q} \|g\|_{L^{\tilde q}}
\] 
To conclude the proof, one uses the fact that $\eta_\tau(x) = \tau^{-d} \eta_1(x/\tau)$, which implies that $\bar\eta_\tau(x) = \tau^{-d -|c|} \bar \eta_1(x/\tau)$ and
$\|\bar \eta_\tau\|_{L^r} = \tau^{-d -|c| + d/r}\|\bar \eta_1\|_{L^r}$ for $r\geq 1$.
\end{proof}
We will also need similar properties with $C^k_0$ instead of $W^{k,q}$.
\begin{Lemma}
\label{taylorlemma.2}
Let $\tilde q\in [1,\infty)$, $k\geq l\in \mathbb{R}$, and $\tau >0$. Then $f \mapsto \eta_\tau \star f$ is a bounded linear mapping from  $W^{l,\tilde q}(\mathbb{R}^d)$ to $C^k_0(\mathbb R^d)$ with, for some $C>0$,
$$\|\eta_{\tau} \star f\|_{C^k}\leq  C \tau^{-2\gamma}\|f\|_{ W^{l,\tilde q}},$$
$f\in W^{l,\tilde q}(\mathbb{R}^d)$ and where 
$$\gamma  = \dfrac{d}{2\tilde q}+\dfrac{1}{2}(k-l).$$

It is also a bounded linear mapping from  $C^l_0(\mathbb{R}^d)$ to $C^k_0(\mathbb R^d)$ with, for some $C>0$,
$$\|\eta_{\tau} \star f\|_{C^k}\leq  C \tau^{-2\gamma}\|f\|_{C^l},$$
$f\in C^l(\mathbb{R}^d)$ and where 
$$\gamma  = \dfrac{1}{2}(k-l).$$
\end{Lemma}
\begin{proof}
The proof is similar to that above, essentially taking $q = \infty$, and $\tilde q = \infty$.

\end{proof}

\subsection{Main result}
 \label{sec:definitions}
\begin{Definition}
Define the $C^1$ function $I:\mathbb{R}\to \mathbb{R}$,  for some  $W>0$ , by
\begin{equation}
\label{eq:I}
I(\xi) =  W \xi^2(\xi-1)^2(\xi-1/2)\cdot \mathbf{1}_{[0,1]}(\xi)
\end{equation}
Let $p\in (2,\infty]$, $r\in (\dfrac{dp}{p-2},\infty]$, and $\sigma>0$ be fixed real numbers.     All constants considered below may depend of these numbers and on $W$. 
    
Let  $\mathcal{U} = L^{p}([0,1],L^{r}(\mathbb{R}^d))$, and $\mathcal{V}= L^p([0,1],V)$. 
\end{Definition}

Our main evolution equation is the following generalized convective Allen-Cahn equation. It is a semilinear parabolic PDE whose mild solutions can be written in  integral form by using the heat kernel $\eta_{\sigma\sqrt{t}}$   as  a contraction semigroup generator; see \cite{taylor}, Chapter 15.1, \cite{EvansPDE} 7.4,\cite{InfDimPMP} Chapter 4, \cite{henry81:GTS}).
\begin{Theorem}\label{GlobalExistence2}
   Suppose $f_0 \in L^{r^*}(\mathbb{R}^d)$,  $u\in \mathcal{U} = L^{p}([0,1],L^{r}(\mathbb{R}^d))$, and $v\in \mathcal{V}= L^p([0,1],V)$. There is a unique mild solution $f\in C^{0}([0,1],L^{r^*}(\mathbb{R}^d))\cap C^{0}((0,1],W^{1,{r^*}}(\mathbb{R}^d))$ to the differential equation 
\begin{equation}
\label{eq:main.pde}
\partial_t f(t,x) = \frac{\sigma^2}{2}\Delta f(t,x)+ I(f(t,x)) + u(t,x)|\nabla f(t,x)| - v(t,x)^T \nabla f(t,x),\quad (t,x)\in [0,1]\times\mathbb{R}^d
\end{equation}
given by the integral equation 
$$f(t) = \eta_{\sigma \sqrt{t}}\star f_0 + \int_0^t \eta_{\sigma \sqrt{t-s}}\star \left(I(f(s)) + u(s)|\nabla f(s)| - v^T(s)\nabla f(s)\right)ds,$$
which satisfies
$$\|f\|_P = \max\left(\sup_{t\in [0,1]}\|f(t)\|_{L^{r^*}},\sup_{t\in (0,1]}\sqrt{t}\|f(t)\|_{W^{1,r^*}}\right)<\infty.$$  Moreover, $\lim_{t\to 0} f(t) = f_0$ in $L^{{r^*}}(\mathbb{R}^d)$. 
\end{Theorem}

\Cref{eq:main.pde} provides our control system, with state $f$ and control $(u,v)\in \mathcal U\times \mathcal V$. The following highlights the specific contributions of its terms.

\begin{enumerate}[label= {\bf \arabic*.},wide]

\item We first note that, in the absence of control (when $u$ and $v$ are identically zero) and when $f_0$ is a characteristic function of some finite-measure $V\subset \mathbb{R}^d$, the function $f$ evolves according to an Allen-Cahn equation, 
$$\partial_t f = \dfrac{\sigma^2}{2}\Delta f + I(f),\quad f_0 = \mathbf{1}_V.$$
This equation approximates a mean curvature flow in the context of phase-field dynamics \cite{NSFPFOverview} as $W\to \infty$ in \cref{eq:I}.  The typical Allen-Cahn equation uses the negative of the derivative of the double well  polynomial $\phi(\xi) = W \xi^2(\xi-1)^2$  (for some $W>0$) in order to push the phase-field toward the values $0$ or $1$. Here $I(\xi)$ has the same sign as $-\phi'(\xi)$ for $\xi\in [0,1]$ and is $C^1$ and compactly supported, which provides numerical stability and an easier proof of \cref{GlobalExistence2}.  

\item This approximate mean-curvature flow regularizes our ``ideal'' control equation 
\begin{equation}
\label{eq:ideal.system}
\partial_t f + \nabla f^T v = |\nabla f| u
\end{equation}
in which the l.h.s. is equal to $\partial_t(f\circ \varphi_{0t}^v) \circ (\varphi_{0t}^v)^{-1}$ where $\varphi_{0t}^v$ is the flow associated with $v$,  and the r.h.s. defines a normal flow of the level sets of $f$, with velocity $u$, which makes possible topological changes. For this reason, we will refer to $v$ as the diffeomorphic control and $u$ as the topological control. \Cref{eq:ideal.system} is reminiscent of the metamorphosis equation, as introduced in \cite{miller2001group,trouve2005metamorphoses,trouve2005local}, that takes the form
\[
\partial_t f + \nabla f^T v = u
\]
with $u\in L^2([0,1], L^2)$. This equation was introduced for image alignment and, unlike \cref{eq:ideal.system}, the contribution of $u$ does not have a clear geometric interpretation. While  the existence of solutions for the metamorphosis system is easily deduced from $\partial_t(f\circ \varphi) = u\circ \varphi$, the analysis of \cref{eq:ideal.system} for general $u\in \mathcal U$ is, up to our knowledge, open, without the introduction of a regularizing term.

\item If we only set $u$ to be identically zero and $f_0$ to be a characteristic function, then \cref{eq:main.pde} is a
a convective Allen-Cahn equation
$$\partial_t f = \dfrac{\sigma^2}{2}\Delta f + I(f) - \nabla f^T v,\quad f_0 = \mathbf{1}_V$$
which has recently been applied in elastic and viscoelastic surfaces in fluids \cite{KLOPPE2024117090}. Non-convergence results for this equation have been obtained if one couples it with a Navier-Stokes method \cite{abels2022non}, and  it has been compared to convection via level-set methods where one resets an evolving distance function \cite{GRAVE2022105569}.

\item Suppose one sets $W=0$, effectively setting $I = 0$, and $f_0 = d_{V^c}$ (the complement distance function of some bounded $V\subset \mathbb{R}^d$). Then \cref{eq:main.pde} becomes
$$\partial_t f = \dfrac{\sigma^2}{2}\Delta f + |\nabla f|u - \nabla f^T v.$$
In this setting if the controls $u,v$ were continuous in $(t,x)$ and bounded, then this equation would be a vanishing viscosity method \cite{Fleming}. Moreover, if $u$ were continuous and bounded and we sent $\sigma \to 0$ and if our initial data was Lipschitz continuous then we would be in the realm of viscosity solutions to Hamilton-Jacobi-Bellman PDEs \cite{OCVSHJPDE,VSParaTimeDiscont}. This method found its inception in \cite{OSHER}  and the most widely used implementation is the ENO/WENO numerical scheme \cite{SHU1,SHU2}.

\end{enumerate}

 For the proof of  existence and uniqueness of mild solutions, we loosely follow the work \cite{Weissler1979,Weissler1980,Weissler1981}, and the outline in \cite{taylor}, Chapter 15.  Our problem differs from those studied there through the introduction of two integrable time-dependent terms, one of which is nonlinear; this results in a Lipschitz constant that is not uniform in $t$ in our contraction mapping proof, requiring finer estimates. In addition, our case of  $u\in L^{p}([0,1], L^r(\mathbb{R}^d))$ is a weaker assumption than in the work  \cite{Lady,First_Initial_BVP} of quasilinear parabolic PDEs.   

\section{Optimal control problem}
\label{sec:opt.control}
Our objective function is defined as follows. Let $f_0,f_{{\rm target}}\in L^{r^*}$ and suppose we have a  $C^1$ functional $U:W^{1,r^*}\to [0,\infty)$ that is bounded below such that $U(f_{{\rm target}}) = 0.$ We fix a positive parameter $C_{\mathrm{top}}$ and minimize
\begin{align*}
&E(u,v) = C_{\mathrm{top}}\int_{0}^1\|u(t)\|^p_{L^r} + \|v(t)\|^p_{V}dt + U(f(1))\text{ over }  (u,v)\in \mathcal{U}\times\mathcal{V}\\
  &\text{subject to }
  f(1) = \eta_{\sigma}\star f_0 + \int_0^1 \eta_{\sigma \sqrt{1-s}}\star\left( I(f(s))+ u(s)|\nabla f(s)| - v(s)^T\nabla f(s)\right)ds
\end{align*}
If such a minimum exists, denote it as $\rho_{\sigma}(f_0,f_{{\rm target}})$. If the minimum exists when we reverse the roles of $f_{{\rm target}}$ and $f_0$, we define the $\sigma$-discrepancy measure by
$$d_{\sigma}(f_0,f_{{\rm target}}) = \min\left(\rho_{\sigma}(f_0,f_{{\rm target}}),\rho_{\sigma}(f_{{\rm target}},f_0)\right).$$
By definition, $d_{\sigma}(f_0,f_{{\rm target}})$ is symmetric and for choice $U$ it also has the property that $f_0 = f_{{\rm target}}\implies d_{\sigma}(f_0,f_{{\rm target}})=0.$ However, we do not claim it satisfies the triangle inequality. 
\par As we show in the coming section, we require $\sigma>0$ to  have existence and uniqueness of mild solutions so that the problem may be well posed. We display numerical simulations in the last section of this paper.

In our experiments, we will compare the numerical minimums of $E$ as we send $C_{\mathrm{top}}$ to infinity. We do this in order to qualitatively study how much of the alignment of $f_0$ and $f_{{\rm target}}$ can be contributed to the large diffeomorphic deformation geometric change caused by $v$ and the non-smooth, geometric and topological change caused by $u$. We consider the parabolic orbits in $W^{1,r*}$ of the smoothed characteristic function that are smoothly convected by $v$, non-smoothly convected by $u$, and smoothed out by the semi-group generator $\eta_{\sigma}$. This is in contrast to LDDMM, where one only seeks a minimal  $v\in \mathcal{V}$  to  align the initial volume $V_0\subset \mathbb{R}^d$ to a target volume $V_{{\rm target}}\subset \mathbb{R}^d$ via the orbit $1_V\mapsto 1_{(\varphi_{01}^v)^{-1}(V)}$. 

 Due to $I$ and $\eta_{\sigma}$, the initial data is also undergoing an approximate mean curvature flow (\cite{NSFPFOverview}), so the shape also changes slightly ``for free." So, all things considered, we are combining LDDMM, level set methods, and phase-field dynamics by using this generalized convective Allen-Cahn evolution equation for shape alignment via a soft-end point optimal control problem.

\section{Existence and Uniqueness of Mild Solutions}
\label{sec:main_section}

We prove \cref{GlobalExistence2} in the following way. 
\begin{enumerate}
    \item Prove the existence of a ``short-time" solution for small enough $T\in (0,1]$ (\cref{localtime} below).
    \item Show that a uniform partition of $[0,1]$ exists with partition mesh $\Delta t< T$ that satisfies a condition that depends on the norms of $u,v$. 
    \item Stitching several ``short-time" fixed point solutions of a single $u,v$ to create the global solution.  
\end{enumerate}


We will use \cref{taylorlemma} in the following special cases that we state as a corollary for easier reference. 
\begin{Corollary}
\label{cor.taylor} 
We let $\epsilon\in \{0,1\}$.
\begin{subequations}
\begin{enumerate}[label=(\roman*)]
\item There exists $C_0>0$ (that depends on $d$ and $\sigma$) such that, for all $f\in L^{r^*}$,
\begin{equation}
\label{eq:taylor.0}
\|\eta_{\sigma \sqrt t} \star f\|_{ W^{\epsilon,r^*}} \leq C_0 t^{-\epsilon/2} \|f\|_{L^{r^*}}
\end{equation}

\item There exists $C_1>0$ (that depends on $d$ and $\sigma$) such that, for all $f\in L^{1}$,
\begin{equation}
\label{eq:taylor.1}
\|\eta_{\sigma \sqrt t} \star f\|_{W^{\epsilon,r^*}} \leq C_1 t^{-\epsilon/2-\gamma_1} \|f\|_{L^{1}}
\end{equation}
with $\gamma_1 = dr/2$. 
\end{enumerate}
\end{subequations}
\end{Corollary}

We will use the following notation and remarks.
\begin{enumerate}[label=\textbf{\arabic*.},wide]
    \item Let $C_I$ denote the Lipschitz constant of $I$ (one can take $C_I = W/16$).  Let $C_V$ be a constant for which 
    $$\|w\|_{C^1}\leq C_V \|w\|_{V},\quad \forall w\in V.$$
    \item One has, for $0< a,b < 1$,
$$\int_0^t (t-s)^{-a} s^{-b}ds = \dfrac{\Gamma(1-a)\Gamma(1-b)}{\Gamma(2-(a+b))} t^{1-(a+b)}< \infty$$
where $\Gamma$ is the classical Gamma function. The integral is unbounded as $a \to 1$ or $b\to 1$. Also remark that $\int_0^t (t-s)^{-\gamma} ds = t^{1-\gamma}/(1-\gamma)$ for $\gamma <1$.
    
   \item
    Let $\phi:[0,1]\times (0,1)\to \mathbb{R}$ be defined as 
    \begin{align*}
        \phi(t,\gamma) &= \left(\int_0^t (t-s)^{-\gamma p^*}s^{-\frac{p^*}{2}}ds\right)^{1/p^*}\\
                       &= \left(\dfrac{\Gamma(1-\gamma p^*)\Gamma(1-p^*/2)}{\Gamma(2-(\gamma+1/2)p^*)}\right)^{1/p^*} t^{1/p^*-\gamma - 1/2}.
    \end{align*}
Remark that there is a $K>0$ such that
    $$\max_{t\in [0,1]}\left(\phi(t,1/2),\sqrt{t}\phi(t,\gamma_1), \sqrt{t}\phi(t,\gamma_1+1/2)\right)<K,\quad C_V \left(\dfrac{1}{1-p^*/2}\right)^{1/{p^*}} < K.$$
    Essential for this condition to be true is that $\dfrac{1}{p^*} - \dfrac{d}{2r} - \dfrac{1}{2} > 0$ which is true when $r > \dfrac{dp}{p-2}$ and $p>2$, as assumed in \cref{GlobalExistence2}.

\item For $T\in (0, 1]$, define $P_T$ as the vector space of functions $f\in C^0([0,T],L^{r^*}(\mathbb{R}^d))\cap C^0((0,T],W^{1,{r^*}}(\mathbb{R}^d))$ such that 
\[
\|f\|_{P_T} := \max\left(\sup_{t\in [0,T]}\|f(t)\|_{L^{r^*}},\sup_{t\in (0,T]}\sqrt{t}\|f(t)\|_{W^{1,r^*}}\right)<\infty.
\]
We note that $(P_T, \|\ \|_{P_T})$ is a Banach space and will use the notation
$$d(f,g)_{P_{T}} = \|f-g\|_{P_T}.$$
\end{enumerate}

We will prove local existence of solutions as stated in the following theorem.
\begin{Theorem}\label{localtime}
      Suppose $f_0 \in L^{r^*}(\mathbb{R}^d)$,  $u\in \mathcal{U}$, and $v\in \mathcal{V}$. For small enough $T\in (0,1]$, there is a unique solution $f\in C^{0}([0,T],L^{r^*}(\mathbb{R}^d))\cap C^{0}((0,T],W^{1,{r^*}}(\mathbb{R}^d))$ to the integral equation 
\begin{equation}
\label{eq:main.equation}
f(t) = \eta_{\sigma \sqrt{t}}\star f_0 + \int_0^t \eta_{\sigma \sqrt{t-s}}\star \left(I(f(s)) + u(s)|\nabla f(s)| - v(s)^T\nabla f(s)\right)ds,\quad \forall t \in [0,T]
\end{equation}
which satisfies
$$\max\left(\sup_{t\in [0,T]}\|f(t)\|_{L^{r^*}},\sup_{t\in (0,T]}\sqrt{t}\|f(t)\|_{W^{1,r^*}}\right)<\infty.$$  Moreover, $\lim_{t\to 0} f(t) = f_0$ in $L^{{r^*}}(\mathbb{R}^d)$ 
\end{Theorem}
\begin{proof}
    Recall that $p\in (2,\infty]$, $r\in (dp/(p-2),\infty]$, and $\sigma>0$. Take  $u\in \mathcal{U}$, $f_0 \in L^{r^*}$, and $v\in \mathcal{V}$. We set up a contraction mapping argument where the mild solution we seek is the unique fixed point of said mapping.

 We fix $T\in (0, 1]$ and define $\Psi:P_T \to P_T$ by 
    $$\Psi(f)(t) = \eta_{\sigma \sqrt{t}}\star f_0 + \int_0^t \eta_{\sigma \sqrt{t-s}}\star \left(I(f(s)) + u(s)|\nabla f(s)| - v^T(s)\nabla f(s)\right)ds.$$

We also introduce the notation, for $u\in \mathcal U$, $v\in \mathcal V$:
\begin{subequations}
\begin{align}
\label{eq:psi.0}
\Psi_0&: t \mapsto \eta_{\sigma \sqrt t} \star f_0 \\
\label{eq:psi.I}
\Psi_I(f)&: t \mapsto \int_0^t \eta_{\sigma \sqrt{t-s}} \star I(f(s)) ds \\
\label{eq:psi.u}
\Psi_u(f)&: t \mapsto \int_0^t \eta_{\sigma \sqrt{t-s}} \star (u(s)|\nabla f(s)|) ds \\
\label{eq:psi.v}
\Psi_v(f)&: t \mapsto \int_0^t \eta_{\sigma \sqrt{t-s}} \star (v(s)^T\nabla f(s)) ds ,
\end{align}    
\end{subequations}
so that $\Psi(f) = \Psi_0 + \Psi_I(f) + \Psi_u(f) - \Psi_v(f)$.

\subsubsection*{Preliminary lemma}

We first prove the following lemma and its corollary.
\begin{Lemma}
\label{lem:ineq.psi}
\begin{subequations}
Let $\epsilon \in \{0,1\}$. For $f,g \in C^0([0,T], L^{r^*}) \cap  C^0((0,T], W^{1,r^*})$, one has, for all $t\in (0, T]$,
\begin{align}
\label{eq:ineq.I}
&\int_0^t \|\eta_{\sigma \sqrt{t-s}}\star (I(f(s))-I(g(s)))\|_{W^{\epsilon, r^*}} ds  \leq C_0C_I \int_0^t (t-s)^{-\epsilon/2}\|f(s) - g(s)\|_{L^{r^*}}ds\\
\label{eq:ineq.u}
&\int_0^t \|\eta_{\sigma\sqrt{t-s}} \star (u(s) (|\nabla f(s)| - |\nabla g(s)|))\|_{W^{\epsilon, r^*}} ds \leq \\
\nonumber
&\qquad\qquad\qquad C_1K\left(\int_0^t\|u(s)\|^{p}_{L^{r}} s^{p/2} \|f(s) - g(s)|\|^p_{W^{1,r^*}}ds\right)^{1/p}\\
\label{eq:ineq.v}
&\int_0^t \|\eta_{\sigma\sqrt{t-s}} \star (u(s)^T \nabla f(s))\|_{W^{\epsilon, r^*}} ds \leq \\
\nonumber
&\qquad\qquad\qquad C_0 C_VK \left(\int_0^t\|v(s)\|_V^{p} s^{p/2}\|f(s) - g(s)|\|^p_{W^{1,r^*}} ds\right)^{1/p}
\end{align}
\end{subequations}
\end{Lemma}

\begin{Corollary}
\label{cor:ineq.psi}
The mapping $\Psi$ maps $P_T$ into itself. Moreover,
given $f,g\in P_T$, one has
\begin{subequations}
\begin{align}
\label{eq:ineq.c0}
\|\Psi_0\|_{P_T} &\leq C_0 \|f_0\|_{L^{r^*}}\\
\label{eq:ineq.cI}
\|\Psi_I(f)-\Psi_I(g)\|_{P_T} &\leq 2C_0C_I T \|f-g\|_{P_T} \\
\label{eq:ineq.cu}
\|\Psi_u(f)-\Psi_u(g)\|_{P_T} &\leq C_1K\|f-g\|_{P_T}  \left(\int_0^T\|u(s)\|^{p}_{L^{r}}ds\right)^{1/p}\\
\label{eq:ineq.cv}
\|\Psi_v(f)\|_{P_T} &\leq C_0C_VK\|f\|_{P_T}  \left(\int_0^T\|v(s)\|_V^{p}ds\right)^{1/p}
\end{align}
\end{subequations}
\end{Corollary}

\begin{proof}[Proof of \cref{lem:ineq.psi} and \cref{cor:ineq.psi}]

Take $\epsilon\in \{0,1\}$. 
Using \cref{eq:taylor.0}, we have $\Psi_0(t) = \eta_{\sigma \sqrt t} \star f_0 \in W^{\epsilon, r^*}$ with
\[
\|\Psi_0(t)\|_{W^{\epsilon, r^*}} \leq C_0 t^{-\epsilon/2} \|f_0\|_{L^{r^*}},
\]
 which proves \cref{eq:ineq.c0}.

Switching to $\Psi_I$, we have, for all $s\in [0, T]$, 
\[
|I(f(s)) - I(g(s))| \leq C_I |f(s) - g(s)|
\]
so that $I(f(s)) - I(g(s)) \in L^{r^*}$  and, using \cref{eq:taylor.0} again, 
$\eta_{\sigma \sqrt{t-s}} \star (I(f(s))-I(g(s))) \in W^{\epsilon, r^*}$ with
\begin{align*}
\|\eta_{\sigma \sqrt{t-s}}\star (I(f(s))-I(g(s)))\|_{W^{\epsilon, r^*}} &\leq 
C_0(t-s)^{-\epsilon/2}\|I(f(s)) - I(g(s))\|_{L^{r^*}}\\
& \leq 
C_0C_I (t-s)^{-\epsilon/2}\|f(s) - g(s)\|_{L^{r^*}}.
\end{align*}
This shows that
\begin{align*}
\|\Psi_I(f(t))-\Psi_I(g)(t)\|_{W^{\epsilon, r^*}} &\leq \int_0^t \|\eta_{\sigma \sqrt{t-s}}\star (I(f(s))-I(g(s)))\|_{W^{\epsilon, r^*}}ds\\
& \leq C_0C_I \int_0^t (t-s)^{-\epsilon/2}\|f(s) - g(s)\|_{L^{r^*}}ds,
\end{align*}
which proves \cref{eq:ineq.I}, and \cref{eq:ineq.cI} is derived from $\|f(s) - g(s)\|_{L^{r^*}} \leq \|f-g\|_{P_T}$ and 
\[
\int_0^t (t-s)^{-\epsilon/2} = t^{1-\epsilon/2}/(1-\epsilon/2).
\]

Switching to $\Psi_u$, we have, for almost all $s\in [0, T]$, 
$u(s) |\nabla f(s)| \in L^1$ with
\[
\|u(s) |\nabla f(s)|\|_{L^1} \leq \|u(s)\|_{L^{r}} \|f(s)\|_{W^{1, r^*}}.
\]
It follows from  \cref{eq:taylor.1} that 
\[
\eta_{\sigma\sqrt{t-s}} \star (u(s) (|\nabla f(s)| - |\nabla g(s)|)) \in W^{\epsilon, r^*}
\]
with 
\begin{align*}
\|\eta_{\sigma\sqrt{t-s}} \star (u(s) (|\nabla f(s)| - |\nabla g(s)|))\|_{W^{\epsilon, r^*}} &\leq C_1 (t-s)^{-\epsilon/2-\gamma_1}\|u(s)(|\nabla f(s)| - |\nabla g(s)|)\|_{L^{1}}\\
&\leq C_1 (t-s)^{-\epsilon/2-\gamma_1}\|u(s)(|\nabla f(s) - \nabla g(s)|)\|_{L^{1}}\\
&\leq C_1 (t-s)^{-\epsilon/2-\gamma_1}\|u(s)\|_{L^{r}} \|f(s) - g(s)|\|_{W^{1,r^*}}\\
\end{align*}
It follows that
\begin{align*}
\|\Psi_u(f)(t)-\Psi_u(g)(t)\|_{W^{\epsilon, r^*}} & 
\leq C_1 \int_0^t(t-s)^{-\epsilon/2-\gamma_1}s^{-1/2} \|u(s)\|_{L^{r}} s^{1/2} \|f(s) - g(s)|\|_{W^{1,r^*}} ds\\
&\leq C_1\phi(t,\epsilon/2+ \gamma_1) \left(\int_0^t\|u(s)\|^{p}_{L^{r}} s^{p/2} \|f(s) - g(s)|\|^p_{W^{1,r^*}} ds\right)^{1/p}\\
&\leq C_1K \left(\int_0^t\|u(s)\|^{p}_{L^{r}} s^{p/2} \|f(s) - g(s)|\|^p_{W^{1,r^*}}ds\right)^{1/p},
\end{align*}
 which yields \cref{eq:ineq.u}, and \cref{eq:ineq.cu} is a consequence of 
$s^{1/2} \|f(s) - g(s)|\|_{W^{1,r^*}} \leq \|f-g\|_{P_T}$.

Similarly, for almost all $s\in [0, T]$, we have $v(s)^T \nabla f(s) \in L^{r^*}$ with
\[
\|v(s)^T \nabla f(s)\|_{L^{r^*}} \leq \|v(s)\|_{C^0} \|f(s)\|_{W^{1, r^*}} \leq C_V \|v(s)\|_{V} \|f(s)\|_{W^{1, r^*}}.
\]
Using  \cref{eq:taylor.0}, we get 
\[
\eta_{\sigma\sqrt{t-s}}\star (v(s)^T \nabla f(s)) \in W^{\epsilon, r^*}
\]
with
\[
\|\eta_{\sigma\sqrt{t-s}}\star (v(s)^T \nabla f(s))\|_ {W^{\epsilon, r^*}}
\leq C_0  C_V \|v(s)\|_{V}  (t-s)^{-\epsilon/2}  s^{-1/2} s^{1/2} \|f(s)\|_{W^{1, r^*}}.
\]
Finally, 
    \begin{align*}
    \|\Psi_v(f)(t)\|_{W^{\epsilon/2,r^*}} &
   \leq C_0 C_V  \phi(t, \epsilon/2)  \left(\int_0^t \|v(s)\|^{p}_{V} s^{p/2} \|f(s)\|^p_{W^{1, r^*}} ds \right)^{1/p} \\
        &\leq C_0 C_VK  \left(\int_0^t \|v(s)\|^{p}_{V} s^{p/2} \|f(s)\|^p_{W^{1, r^*}} ds\right)^{1/p} ,
    \end{align*}
    which  gives \cref{eq:ineq.v,eq:ineq.cv}.

If $f\in P_T$, these equations (taking $g=0$) imply that $\sup_{t\in (0,T]} t^{\epsilon} \|\Psi(f)(t)\|_{W^{1, r^*}} < \infty$. Obviously, $\Psi(f) \in C^0((0, T], W^{\epsilon, r^*})$. But \crefrange{eq:ineq.cI}{eq:ineq.cv} show that $\Psi(f)(t)$ has the same limit (in $L^{r^*}$) as $\eta_{\sigma\sqrt t} \star f_0$ when $t\to 0$, which is equal to $f_0$ from standard properties of mollifiers, showing that $\Psi(f) \in C^0([0, T], L^{r^*})$. 
\end{proof}
    
\subsubsection*{Well defined contraction}
A direct  consequence of \cref{cor:ineq.psi} is that, if $T$ is small enough, such that, say, 
\begin{equation}
\label{eq:contraction}
2C_0 C_I T + 
C_2K \left(\int_0^T\|u(s)\|^{p}_{L^{r}}ds\right)^{1/p} +
C_0C_VK  \left(\int_0^T\|v(s)\|_V^{p}ds\right)^{1/p} < \frac 12
\end{equation}
then $\Psi$ is a contraction on $P_T$. 
Choosing such a $T$, we obtain the fact that there is a unique solution 
to \cref{eq:main.equation} in $P_T$.

\subsubsection*{Arbitrary integral solution is element of path space}
We show that if a function $f\in C^0([0,T'], L^{r^*}) \cap  C^0((0,T'], W^{1,r^*}) $ with $0<T\leq T'\leq 1$ is such that 
$$f(t) = \eta_{\sigma \sqrt{t}}\star f_0 + \int_0^t \eta_{\sigma \sqrt{t-s}}\star \left(I(f(s)) + u(s)|\nabla f(s)| - v^T(s)\nabla f(s)\right)ds$$
then $f|_{[0,T]} \in P_{T}$, which implies that the function $f$ is the unique fixed point of $\Psi:P_{T}\to P_{T}$. This boils down to showing that $\sqrt{t} \|f(t)\|_{W^{1, r^*}}$ remains bounded when $t \to 0$. 

Let $A_t = \max_{0<s\leq t} (\sqrt{s} \|f(s)\|_{W^{1, r^*}})$. Using \cref{eq:ineq.c0} and \crefrange{eq:ineq.I}{eq:ineq.v}, we get, for $0\leq s \leq t \leq T$:
$\sqrt s \|\Psi_0(s)\|_{W^{1, r^*}} \leq C_0 \|f_0\|_{L^{r^*}}$ and
\begin{align*}
\sqrt s \|\Psi_I(f)(s)\|_{W^{1, r^*}}& \leq C_0 C_I \sqrt s \left(\int_0^s(s-\tilde s)^{-1/2} \tilde s^{-1/2} d\tilde s\right) A_s  = C_0 C_I \sqrt s \Gamma(1/2)^2 A_s\\
\sqrt s \|\Psi_u(f)(s)\|_{W^{1, r^*}} & \leq 
C_1K\left(\int_0^s\|u(\tilde s)\|^{p}_{L^{r}} d\tilde s\right)^{1/p} A_s\\
\sqrt s \|\Psi_v(f)(s)\|_{W^{1, r^*}} & \leq
C_VK \left(\int_0^s\|v(\tilde s)\|_V^{p} d\tilde s\right)^{1/p} A_s
\end{align*}
Since $f(t) = \Psi(f)(t)  = \Psi_0 + \Psi_I(f)(t) + \Psi_u(f)(t) - \Psi_v(f)(t)$, we get (taking the supremum over $s\leq t$)
\[
A_t \leq C_0 \|f_0\|_{L^{r^*}} + \left(C_0 C_I \Gamma(1/2)^2 \sqrt t+
C_1K\left(\int_0^t\|u(s)\|^{p}_{L^{r}} ds\right)^{1/p}
+ C_VK \left(\int_0^t\|v(s)\|_V^{p} ds\right)^{1/p}\right) A_t.
\]
Taking $t_0$ small enough so that
\[
C_0 C_I \Gamma(1/2)^2 \sqrt {t_0}+
C_1K\left(\int_0^{t_0}\|u(s)\|^{p}_{L^{r}} ds\right)^{1/p}
+ C_VK \left(\int_0^{t_0}\|v(s)\|_V^{p} ds\right)^{1/p}\leq \frac12
\]
we get $A_t \leq 2C_0 \|f_0\|_{L^{r^*}}$ for all $t\leq t_0$, showing that $f\in P_T$.

\end{proof}

\subsubsection*{Global existence}

\begin{proof}[Proof of Theorem \ref{GlobalExistence2}]
Remark that in the previous proof, for any $f_0 \in L^{r^*}$, $u\in \mathcal{U}$, $v\in \mathcal{V}$ the short time existence of a solution to the integral equation depended on picking $T$ such that \cref{eq:contraction} is satisfied. Taking the $p$th power of this equation, and using the inequality $a+b \leq 2^{1/p^*} (|a|^p + |b|^p)$, this equation is satisfied as soon as 
\[
h(T) := 
CT + 
C\int_0^T\|u(s)\|^{p}_{L^{r}}ds +
C\int_0^T\|v(s)\|_V^{p}ds \leq \frac12,
\]
with $C = \max(2C_0C_I, 2^{1/p^*} C_1K, 2^{1/p^*} C_0C_VK)$. Since $h$
is continuous in $T$ over $[0,1]$,  there exists (by uniform continuity) a partition $0=t_0<t_1<...<t_N = 1$ with fixed mesh size $\Delta t = t_2 - t_1$ on $[0,1]$ so that for each $i = 0,...,N-1$ we have $|h(t_{i+1}) - h(t_i)| \leq 1/2$, i.e., 
   \begin{equation}\label{cond1}
C\Delta t+      C\int_{t_i}^{t_{i+1}} \|u(s)\|_{L^r}^p ds + C\int_{t_i}^{t_{i+1}}\|v(s)\|_{V}^p ds < \dfrac{1}{2}.
   \end{equation}
We define the functions $u_i\in \mathcal{U},v_i\in \mathcal{V}$ given for each $i=0,...,N-1$ by 
    $$u_i(t) = u(t + t_i),\, t\in [0,
    \Delta t],\quad u_i(t) = 0,\, t>\Delta t,$$
    $$v_i(t) = v(t + t_i),\, t\in [0,
    \Delta t],\quad v_i(t) = 0,\, t>\Delta t.$$
Given Equation \ref{cond1} and the fact that 
$$\int_{t_{i}}^{t_{i+1}} \|u(s)\|_{L^r}^p + \|v(s)\|_{V}^p ds = \int_{0}^{\Delta t} \|u_i(s)\|_{L^r}^p + \|v_i(s)\|_{V}^p ds$$ 
for each $i$, then each $u_i,v_i$ emit a local integral solution for any given initial data inside of $ L^{r^*}$. So let $f^{1}$ denote the unique solution to the integral equation using $u_0,v_0,f_0$. We then solve the integral equation $N-1$ times iteratively for each $i = 1,...,N-1$. That is, we define $f^{i+1}$ as the short time solution of the integral equation using  using $u_i,v_i,f_i$.
  Then define $f\in C^{0}([0,1],L^{r^*}(\mathbb{R}^d))\cap C^{0}((0,1],W^{1,{r^*}}(\mathbb{R}^d))$ by 
   $$f(t) = f^i(t-t_{i}),\,t\in [t_{i-1},t_{i}],i=1,...,N.$$
   By definition for $t\in [t_{i},t_{i+1}]$, this $f$ satisfies the condition that 
   \begin{align*}
       f(t) &= \eta_{\sigma \sqrt{t-t_i}}\star f^{i}(\Delta t) +\int_{t_i}^{t_{i+1}}\eta_{\sigma \sqrt{t-s}} \star\left( I(f^{i+1}(s)) + u_i(s)|\nabla f^{i+1}(s)|- v_i(s)^T\nabla f^{i+1}(s)\right)ds\\
       &= \eta_{\sigma \sqrt{t}}\star f(t) +\int_{0}^{t}\eta_{\sigma \sqrt{t-s}} \star \left(  I(f(s)) +u(s)|\nabla f(s)|- v(s)^T\nabla f(s)\right)ds,
   \end{align*}
which can be shown by using the following iteration, convolution properties of Gaussians,  and the definitions of $u_i,v_i$
   $$f^{i}(\Delta t) = \eta_{\sigma \sqrt{\Delta t}}\star f^{i-1}(\Delta t) +\int_{0}^{\Delta t}\eta_{\sigma \sqrt{\Delta t-s}} \star \left( I(f^{i}(s) + u_{i-1}(s)|\nabla f^{i}(s)|-v_{i-1}(s)^T\nabla f^{i}(s)\right)ds.$$
   This function $f$ is the only function that satisfies this global integral equation since at each time $t$ we have $f(t) = f^i(t-t_{i})$ where $f^i(t-t_{i})$ uniquely solves the local integral equation. Moreover
    \begin{multline*}
        \max\left(\sup_{t\in [0,1]}\|f(t)\|_{L^{r^*}} ,\sup_{t\in (0,1]}\sqrt{t}\|f(t)\|_{W^{1,r^*}} \right) \leq\\
        \max_{i=0,...,N-1}\left(\max(\sup_{t\in [t_{i},t_{i+1}]}\|f^{i+1}(t)\|_{L^{r^*}} ,\sup_{t\in (t_{i},t_{i+1}]}\sqrt{t}\|f^{i+1}(t)\|_{W^{1,r^*}} )\right)<\infty.
    \end{multline*}
    And, lastly, $\lim_{t\to 0}f(t)=f_0$ almost everywhere in $\mathbb{R}^d$ as
    $$\lim_{t\to 0}\|f(t)-f_0\|_{L^{r^*}} = \lim_{t\to 0}\|f^1(t)-f_0\|_{L^{r^*}} = 0.$$
\end{proof}

\section{A discrete-time optimal control problem}
\label{sec:opt.cont}
\subsection{Discrete-time evolution}
Let $T$ be a  positive integer and let $\Delta t = \dfrac{1}{T}$. We discretize $[0,1]$ uniformly into the $T+1$-tuple given by $t_0 = 0$, $t_1 = \Delta t$, $t_2 = 2 \Delta t,..., t_T = T \Delta t = 1$. Let $p\in (2,\infty)$, and $r\in \left(\dfrac{dp}{p-2},\infty\right)$.  We define $P = C([0,1],L^{r^*})\cap C([0,1],W^{1,r^*})$ and
$$Q_{T} = \{f : f:\{1,...,T\} \to C^{1}_0\},$$ 
where $\frac{1}{r} + \frac{1}{r^*}=1$ and identify elements of $Q_{T}$ as $T+1$-tuples with $f_0 \in L^{r^*}$ and $f_i \in C^1_0$ for $i =1,..T$.  Define the subspaces $\mathcal{U}_{T}$ of $\mathcal{U}$ and $\mathcal{V}_{T}$ of $\mathcal{V}$ by
$$\mathcal{U}_{T} = \{u\in \mathcal{U}: u_{[t_{i},t_{i+1})} = u_i, u_i \in L^r\,\forall i = 1,...,T-1,\,u_{[t_{0},t_{1})} = 0\}.$$ 
$$\mathcal{V}_{T} = \{v\in \mathcal{V}: v_{[t_{i},t_{i+1})} = v_i, v_i \in V\,\forall i = 1,...,T-1,\,v_{[t_{0},t_{1})} = 0\}.$$ 

So by definition, the $r$-th, 2-th powers of the norms of $u,v \in \mathcal{U}_{\Delta t},\mathcal{V}_{\Delta t}$ can be written as a finite sum 
$$\|u\|^r_{\mathcal{U}} = \int_0^1 \|u(t)\|^r_{\mathcal{U}}dt = \sum_{t=1}^{T-1}\|u_i\|^r_{L^p}\Delta t,\quad \|v\|^2_{V} = \int_0^1 \|v(t)\|^2_{V}dt = \sum_{t=1}^{T-1}\|v_i\|^2_{V}\Delta t,$$ 
where $u_i = u_{[t_{i},t_{i+1})},v_i = v_{[t_{i},t_{i+1})}$ for all $i = 1,...,T-1$. So we write $u_i,v_i$ to refer to the $i$-th entry of an element of $\mathcal{U}_{T},\mathcal{V}_{T}$ and regard its elements as $T-1$-tuples in $L^{r}$ and $V$. 
We pick some small constant $\psi > 0$, and define 
\[
\chi(z) = \sqrt{|z|^2 + \psi^2} 
\]
as a smooth approximation of $z\mapsto|z|$. We note that the gradient of $\chi$ is $\nabla \chi(z) = \dfrac{z}{\sqrt{|z|^2+ \psi^2}}$, with $|\nabla \chi(z)| \leq 1$ and its Hessian is
\[
\nabla^2 \chi = \frac{(|z|^2 + \psi^2)I_{\mathbb R^d} - zz^T}{(|z|^2 + \psi^2)^{3/2}} \geq\frac{I_{\mathbb R^d}}{\psi}
\]
where $I_{\mathbb R^d}$ is the $d$-dimensional identity matrix and the inequality is meant for the usual partial order between symmetric matrices (positive semi-definite difference). Since $\chi(0) = 0$, we have $\chi(z) \leq |z|$ for all $z$.



Let $\tau = \sigma \sqrt{\Delta t}$. By \cref{taylorlemma.2}, $h \mapsto \eta_\tau \star h$ is a bounded linear mapping from any $L^q$ to $C^1_0$, and from $C^0_0$ to $C^1_0$.  Define $G: C^1_0 \times L^r \times V \to C^1_0$ by 
\begin{equation}
\label{eq:doptc.G}
    G(f',u',v') = \eta_{\tau}\star f' +\eta_{\tau}\star \left(  \chi(\nabla f') u' -  \nabla f'^T v' + I(f')\right).
\end{equation}
  We verify that $G$  is indeed a bounded operator in the following lemma and will end this section by proving $G$ is Fréchet differentiable.
\begin{Lemma}
    We claim $G$ in in \cref{eq:doptc.G} is a bounded operator. That is, for $f'\in C^1_0$, $u'\in L^{r}$, and $v' \in V$ and for some constant $C'>0$ we have
    \begin{equation}\label{disc_ineq2}
    \| G(f',u',v')\|_{C^1}\leq C' \|f'\|_{C^1}(1+  \|u'\|_{L^r} + \|v'\|_{V})
    \end{equation}
\end{Lemma}
\begin{proof}
Using \cref{taylorlemma.2}, $f' \mapsto \eta_\tau \star f'$ is bounded and linear from   $C^1_0$ into $C^1_0$.
We have, for some constant $C$,
    $$ 
    \|\eta_{\tau}\star \left(\chi(\nabla f') u'\right)\|_{C^1} \leq C \|\chi(\nabla f') u'\|_{L^r}.
    $$
    Moreover
    \begin{equation}
    \label{eq:lem.doptc.u}
        \|\chi(\nabla f') u'\|_{L^r} \leq \| \chi(\nabla f')\|_{C^0}|u'\|_{L^r}
        \leq \| f'\|_{C^1}|u'\|_{L^r}.
    \end{equation}
    
    For the convective term, we have, for some constant, $C$,
  \begin{equation}
  \label{eq:lem.doptc.v}
       \|\eta_{\tau}\star (\nabla f'^{T} v')\|_{C^1} \leq     C \|\nabla f'^{T} v'\|_{C^0} 
\leq C\| f'\|_{C^1}\ \|v'\|_{C^0} \\
         \leq C_V C \| f'\|_{C^1}\ \|v'\|_{V}.
    \end{equation}
   Lastly the reaction term is bounded as 
   \[
          \|\eta_{\tau}\star I(f')\|_{C^1}  \leq C  \| I(f')\|_{C^0}\leq C_0C_I   \| f'\|_{C^0}.
          \]
   As we add the terms above, we can bound $\|G(f',u')\|_{C^1}$ as claimed.
  \end{proof}
 We define a discrete-time evolution as follows. 
\begin{Proposition}
Let $f_0 \in L^{r^*}$, $u \in \mathcal{U}_{T}$, and $v\in \mathcal{V}_{T}$. Let $f_1 = \eta_{\tau}\star f_0$ and for each $i=1,...,T-1$ define
$$f_{i+1} = G(f_i,u_i,v_i).$$
Then $f_{1:T}\in Q_{T}$, i.e., $f_i\in C^1_0$ for every $i = 1,...,T$.
\end{Proposition}
\begin{proof}
This is clear, given the previous lemma, since $f_1 = \eta_\tau \star f_0\in C^1_0$. 
\end{proof}
One can show that, for given $f_0 \in Q$ and $u\in \mathcal{U}_{T}$, $v\in \mathcal{V}_{T}$, the function $f$ defined in the proposition above has the property that 
\begin{equation}
\label{eq:dopt.c.f}
    f_{k+1}= \eta_{\tau\sqrt{k+1}}\star f_0  + \sum_{j=1}^{k} \eta_{\tau}\star G(f_i,u_i,v_i) \Delta t,\quad \forall k\geq 2.
\end{equation}
Moreover, one can define a $f\in P$ such that $f(t_k) = f_k$ for every $k = 0,1,...,T$ by defining for $k=0,...,T-1$
$$
f(t_{k} + s ) = \eta_{\sigma\sqrt{s}}\star f_k  +  \eta_{\sigma \sqrt{s}}\star G(f_k,u_k,v_k) s.
$$
\subsection{Variational derivatives of the discrete evolution equation}
In this section we prove the following theorem.
\begin{Theorem}
\label{th:G.C1}
The function $G: C^1_0 \times L^r \times V \to C^1_0$ in \cref{eq:doptc.G} is continuously differentiable in a Fréchet sense. One has
\begin{subequations}
\begin{align}
\label{eq:prt.G.u}
\partial_u G(f,u,v) \delta u &= \eta_\tau \star (\chi(\nabla f) \delta u)\\
\label{eq:prt.G.v}
\partial_v G(f,u,v) \delta v &= -\eta_\tau \star (\nabla f^T \delta v)\\
\label{eq:prt.G.f}
\partial_f G(f,u,v) \delta f &= \eta_\tau \star \delta f + \eta_\tau \star \left(u\nabla\chi(\nabla f)^T\nabla \delta f) - v^T\nabla \delta f + I'(f) \delta f\right)
\end{align}
\end{subequations}
\end{Theorem}

\begin{proof}
From \cref{eq:lem.doptc.u,eq:lem.doptc.v},  $G$ is linear and bounded with respect to $u$ and $v$, so that one only needs to consider its differentiability in $f$

The mapping $f \mapsto \eta_\tau \star f$ is a bounded linear mapping from $C^1_0$ to itself,  hence differentiable. The same is true, using \cref{eq:lem.doptc.v}, for $f \mapsto \eta_{\tau} \star (v^T\nabla f)$ if $v\in V$. 
Let $C'_I$ be an upper-bound of both first and second derivatives of $I$. We have
\[
|I(f+h) - I(f) - I'(f) h| \leq C'_I |h|^2,
\]
Then
\[
\|\eta_\tau \star (I(f+h) - I(f) - I'(f) h)\|_{C^1} \leq C'_I \|h\|_{C^1}^2.
\]
showing that $f\mapsto \eta_\tau \star I(f)$ is differentiable, with differential $h\mapsto \eta_\tau \star(I'(f)h)$.

Finally, for $u\in L^r$, we have 
\[
\|\chi(\nabla f+\nabla h) - \chi(\nabla f) - \nabla\chi(\nabla f)^T \nabla h\|_{C^0} \leq \frac{|\nabla h|^2}{\psi}
\]
and this implies that, for some constant $C$,
\[
\|\eta_\tau\star(\chi(\nabla f+\nabla h) - \chi(\nabla f) - \nabla\chi(\nabla f)^T \nabla h)\|_{C^1} \leq \frac{C}{\psi} \|u\|_{L^r} \|h\|_{C^1}^2.
\]  
\end{proof}

We will need below the expressions of the the adjoints of the partial derivatives of $G$, where $\partial_u G^*$, $\partial_v G^*$, $\partial_f G^*$ and defined on $(C^1_0)^*$ and take values, respectively, in $L^{r^*}$, $V^*$ and $(C^1_0)^*$. For a Banach space $B$, let the pairing of a dual vector $\lambda \in B^*$ and vector $v\in B$ be denoted by $\left( \lambda \mid v\right)$. 
If $\lambda \in (C^1_0)^{*}$, we will use the notation
\[
(\lambda \star \eta_\tau\mid h) = (\lambda \mid \eta_\tau\star h)
\]
and $\lambda \star \eta_\tau$ belongs to $(C^1_0)^*$ and to $L^{q^*}$ for any $q>1$. From eqs. (\ref{eq:prt.G.u}-\ref{eq:prt.G.f}), we get (letting $\nabla\cdot$ denote the divergence operator),
\begin{subequations}
\begin{align}
\label{eq:gstar.u}
\partial_u G(f,u,v)^* \lambda&= \chi(\nabla f) (\lambda \star \eta_\tau)  \\
\label{eq:gstar.v}
\partial_v G(f,u,v)^* \lambda&=  -\nabla f (\lambda \star \eta_\tau)\\
\label{eq:gstar.f}
\partial_f G(f,u,v)^* \lambda&= (1+I'(f))\lambda\star\eta_\tau - \nabla \cdot \left ( (u\nabla\chi(\nabla f) -v) \lambda\star\eta_\tau\right).
\end{align}
\end{subequations}

\subsection{Characterization of Minimizers}
We follow Appendix D.4 of \cite{ShapesDiff} and Chapters 4 and 7 of \cite{InfDimPMP} to formulate a Lagrange multiplier method, which can be used to derive a discrete-time version of Pontryagin's Maximum Principle. For a continuous-time maximum principle, see Appendix D.3.1 of \cite{ShapesDiff} and Chapters 4 and 7 of \cite{InfDimPMP}. 
\begin{Proposition}\label{SoftEnd}
Let $f_0 \in L^{r^*}$ and let $U:C^1_0\to \mathbb{R}$ be continuously Fréchet differentiable, and bounded from below. We seek to minimize $E: Q_{T}\times \mathcal{U}_{T} \times \mathcal V_T \to \mathbb{R}$ given by 
$$
E(f,u,v) = \Delta t \sum_{k=1}^{T-1} (C_{\mathrm{top}}\|u_k\|_{L^r}^p + \|v_k\|_V^p) + U(f_T) $$
subject to the constraint that $f = f(f_0,u,v)$ where
$$f_{k+1} = G(f_{k},u_k,v_k),\, k = 1,...,T-1,\quad f_{1} = \eta_{\tau} \star f_0.$$ 

If $(f,u,v)$ is a minimizer of $E$ subject to this constraint then there is a $\lambda: \{1,...,T\}\to (C^1_0)^*$ such that $(f,u,v,\lambda)$ are critical points of the Lagrangian 
\begin{multline}
\label{eq:lag.softend}
    \mathcal{L}(f,u,v,\lambda) = \sum_{k=1}^{T-1} (C_{\mathrm{top}}\|u_k\|_{L^r}^p + \|v_k\|_V^p)+ U(f_T) 
+    \sum_{k=2}^{T}\left(\lambda_k \mid f_{k} - G(f_{k-1},u_{k-1},v_{k-1})\right) \\
+ \left(\lambda_1 \mid f_1 - \eta_{\sigma \sqrt{\Delta t}} \star f_0\right) 
\end{multline}
\end{Proposition}
\begin{proof}
The mapping $u\mapsto \|u\|^p_{L^r}$ is differentiable on $L^r$ with differential $p \|u\|^{p-r}_{L^r} |u|^{r-1} \mathrm{sign}(u)$ and $v \mapsto \|v\|_V^p$ is differentiable on $V$ with gradient $p \|v\|^{p-2} v$. Since $U$ is assumed to be differentiable, we can conclude that $E$ is differentiable. 

The constraints are $\Gamma (f,u,v) = 0$ where $\Gamma = (\Gamma_1, \ldots, \Gamma_T)$ with 
$$\Gamma_i((f,u,v)) = f_{i} - G(f_{i-1},u_{i-1},v_{i-1}),\,i=2,...,T;\quad \Gamma_1((f,u))= f_1 - \eta_{\tau}\star f_0.$$
\Cref{th:G.C1} states that each $\Gamma_i$ is differentiable, and, since $\Gamma$ is defined through a triangular system, one gets the fact that $\partial_f \Gamma$ is onto. These properties are sufficient for the Lagrange multipliers theorem, as stated in \cref{SoftEnd}, to hold.
   \end{proof} 

We then have the following discrete-time evolution equation for $\lambda$ below. (Compare with (1.8), pp 131 of \cite{InfDimPMP} and D.20, pp of \cite{ShapesDiff}.)
\begin{Corollary}
Critical points of the Lagrangian \eqref{eq:lag.softend} in \cref{SoftEnd} 

are such that the following discrete time evolution holds:
\begin{flalign}\label{de-evolution}
    \lambda_T &= dU(f_T),\quad  \lambda_{k} = -\partial_f G(f_k,u_k,v_k)^* \lambda_{k+1}\quad  k = 1,...,T-1
    \end{flalign}

Moreover, if we set $\tilde{E}:\mathcal{U}_{T}\times \mathcal{V}_{T}\to \mathbb{R}$ as $\tilde{E}((u,v)) = E(f(f_0,(u,v)),(u,v))$, then (setting $f = f(f_0,(u,v))$) we have
\begin{align}
\partial_u \tilde{E}((u,v))_k &= \partial_u \mathcal{L}(f,u,v,\lambda) =  \Delta t C_{\mathrm{top}} p\|u_k\|^{p-r}_{L^r} \cdot u_k^{r-1} - \partial_u G(f_{k},u_k, v_k)^* \lambda_{k+1},\\
\partial_v\tilde{E}((u,v))_k &= \partial_v \mathcal{L}(f,u,v,\lambda) = \Delta t p\|v_k\|^{p-2}_{V} \cdot v_k - \partial_v G(f_{k},u_k,v_k)^* \lambda_{k+1},
\end{align}
for $k=1,...T-1$.
\end{Corollary}

We call the problem described in \cref{SoftEnd} a soft end point minimization optimal control problem.  We use the discrete evolution of $\lambda$ in \cref{de-evolution} in the next numerical results section to numerically find optimal $u,v$ given an initial guess---this is called an adjoint shooting method. 
\section{Numerical simulation}
\label{sec:experiments}
In this section we show some two-dimensional simulations of the spatially discretized version of the above optimal control problem. 

\subsection{Maximum bounded principle enforcing time step }
One valuable property to be preserved by discretizations of Allen-Cahn equations is the maximum bounded principle (MBP), that is, if $-\infty < a = \min_{x\in \mathbb{R}^d}(f_0)$ and $b = \max_{x\in \mathbb{R}^d}(f_0)<\infty$, then $a\leq f(t,x) \leq b$ for all $t\in [0,1]$ and $x\in \mathbb{R}^d$. This is a property enjoyed by the convective Allen-Cahn equation for $u=0$,  \cite{MBP1}  gives an example MBP preserving discretization scheme for this case. The MBP is also enjoyed by a variety of semi-linear parabolic equations; an example of an MBP-preserving discretization scheme is provided by  \cite{MBP2}.

We modify the discrete time evolution presented in the previous section to artificially enforce this property. 

First we present a slightly modify the discrete time evolution, so that we can use properties of the derivatives of convolution. We first remark that the discrete time step $f_{k+1} = G(f_k,u_k,v_k)$ for $k=1,...,T-1$ and $f_1 = \eta_{\sigma \sqrt{t}}\star f_0$ is equal to (still letting $\tau = \sigma \sqrt{\Delta t}$)
\begin{equation}\label{newiter}
 F_{k+1} =  (\eta_{\tau}\star F_k)+  \Delta t \left(\chi(|\nabla \eta_{\tau}\star F_k|) u_k -  \nabla \eta_{\tau}\star F_k^T v_k + I(\eta_{\tau}\star F_k)\right)
\end{equation}
for $k = 1,...,T-1$ and $F_0 = f_0$. The iterations are related by
$f_k = \eta_{\tau} \star F_k$
for all $k=0,...,T$.  We have used the fact that $f_k$, for $k\geq 1$, is in $C^1_0$ so $ \eta_{\tau}\star \nabla f_{k}=\nabla \eta_{\tau}\star f_{k} $, and $\chi$ is the same as before. 

Our spatial discretization in the next section will evaluate $\nabla \eta_{\tau}$ over a matrix of points. But we approximate this forward Euler time step for $F_{k+1}$ in the following way, which enforces the maximum bounded principle; heuristically, we write $F_{k+1} = \tilde{F}_k + \Delta t V_k$. Suppose $a = \inf_{\mathbb{R}^d}(f_0)>-\infty$ and $b= \sup_{\mathbb{R}^d}(f_0)<\infty$. Then we pick a $C^2$, analytic, monotonic increasing function $g: [a,b]\to \mathbb{R}$ and consider the following scheme
\begin{align*}
      F_{k+1} &= g^{-1}\left(g(f_{k}) + \Delta t g'(f_{k}) V_k \right),\\
    V_k& =\chi(|\nabla f_{k}|) -  \nabla f_{k}^T v_k + I(f_{k})\\
    f_k &= \eta_\tau\star F_k.
\end{align*} 
We use a Taylor expansion and the inverse function theorem to show that the update rule is at least as accurate as a first order forward Euler scheme as 
$$g^{-1}\left(g(f_k) + \Delta t g'(f_k) V_k \right) \approx f_k + \Delta t  V_k + O(\Delta t).$$

And by construction, $a\leq F_{k+1}(x),f_{k+1}(x) \leq b$ for all $x\in \mathbb{R}^d$. Note that we do not show that the continuous-time mild solution obeys the maximum bounded principle, but it is a reasonable conjecture and that we aim to tackle in a future work.  So, the Lagrangian we will use is 
\begin{equation}
    \mathcal{L}(F,u,v,\lambda) = \Delta t \sum_{k=1}^{T-1} (C_{\mathrm{top}} \|u_k\|_{L^r}^p + \|v_k\|_{V}^p) + U(\eta_{\tau}\star F_{T}) + \sum_{k=1}^{T-1}\left(\lambda_{k+1} \mid F_{k+1} -\tilde{G}(F_k,u_k,v_k)\right) 
\end{equation}
where ${F}_k$  and $V_k$ are defined above and 
\[
\tilde{G}(F_k,u_k,v_k) = g^{-1}\left(g(\eta_{\tau}\star F_{k}) + \Delta t g'(\eta_{\tau}\star F_{k}) V_k\right).
\]
 Remark, there is a $f\in P$ such that $f(t_k) = f_k$ for every $k = 0,1,...,T$. For $k=0,...,T-1$ and $s \in [0,\Delta t]$ it is given by 
$$
f(t_{k} + s ) =  \eta_{\sigma \sqrt{s}}\star g^{-1}\left(g(f_{k}) + \tau g'(f_{k}) V_k\right)$$
which has the property 
$$f(t_{k} + s) =  \eta_{\sigma \sqrt{s}}\star f_{k} + s \eta_{\sigma \sqrt{s}}\star V_k+ o(s).$$

\subsection{Spatial discretization in two dimensions}
\subsubsection{Grid setup, initial and target data}
Let $L>0$, and $N>0$ be a large odd integer and consider the set of points $(x_i,x_j)_{i,j=1}^N$ where $x_i= -L + \Delta x (i-1)$ and $\Delta x = \dfrac{2L}{N-1}$ is the mesh parameter. This will be our discretized $[-L,L]^2$. Let $T>0$ be an integer and consider the set of points $(t_k)_{k=1}^T$ where $t_k = (k-1)\Delta t$ and $\Delta t= \dfrac{1}{T-1}$. Then for any $V_0,V_1\subset [-L,L]^2$ we define our initial data  $f_0$ as the $N\times N$ matrix with entries $[f_0]_{ij}= \mathbf{1}_{V_0}((x_i,x_j))$ and our target data $f_1$ as the $N\times N$ matrix with entries $[f_1]_{ij} =\mathbf{1}_{V_1}((x_i,x_j))$. Moreover, for $\sigma >0$, $\tau = \sigma \sqrt{\Delta t}$, define the $N\times N$ matrices $M$ and $DM$ by their entries
$$[M]_{ij} =  \dfrac{(\Delta x)^2}{\sqrt{2\pi \tau^2 }}\exp\left(-\dfrac{x_i^2 +  x_j^2}{2\tau^2 }\right),\quad [DM]_{ij} =  -\dfrac{x_i}{\tau^2}\cdot [M]_{ij}.$$
We discretize the convolution operators and their derivatives by taking the following discrete linear convolutions
$$
\eta_{\tau}\star h \approx M \star h((x_i,x_j)),\quad \partial_{x_1}\eta_{\tau}\star h \approx DM \star h((x_i,x_j)),\quad \partial_{x_2}\eta_{\tau}\star h \approx DM^T \star h((x_i,x_j)).$$

\subsubsection{Finite Dimensional Subspaces of the Control Spaces }
Subspaces of $L^r$ and $V$ are built on the grid above but in different fashions.  Define the subspace of simple functions
$$L^r_{N} = \{u\in L^r\,:\, u = \sum_{i,j=1}^N u_{ij} \mathbf{1}_{V_{ij}}, u_{ij}\in \mathbb{R},\forall 1\leq i,j\leq N\},$$
$$V_{ij} = \{(x,y): x_{i}\leq x < x_{i+1}\,,x_{j}\leq y < x_{j+1}\}$$
and define the subspace $$\mathcal{U}_{T,N} = \{u\in \mathcal{U}_{T}\,:\,u(t)\in L^r_{N},\, \forall t \in [0,1]\}.$$
We denote the elements of $\mathcal{U}_{T,N}$ by their constant coefficients $u^k_{ij}$ where $1\leq k\leq T$ and $1\leq i,j\leq N$. Remark that, in this case
\begin{equation}\label{unorm}
    \|u\|_{\mathcal{U}}^p = \Delta t  \sum_{k=1}\left((\Delta x)^2\sum_{i,j = 1}(u^k_{ij})^{r}\right)^{p/r},\quad \forall u\in \mathcal{U}_{T,N}.
\end{equation}
The space $V$ of admissible vector fields $v$ is built as a reproducing kernel Hilbert space by selecting a radial kernel function $K:\mathbb{R}^d\times \mathbb{R}^d \to \mathbb{R}$ (i.e., $K(x,y)= \kappa(|x-y|)$ for some function $\kappa$); see \cite{ShapesDiff}, Chapter 8. We then select a finite dimensional subspace $V_{N}\subset V$ defined as the coupled span of kernel functions centered around each point $(x_i,x_j)$:
$$V_{N} = \left[{\rm span}\left(K((\cdot,\cdot),(x_i,x_j))\,:\,i,j=1,...,N\right)\right]^2 $$
so that the grid points $(x_i,x_j)_{i,j=1}^N$ are collocation points of $V_{N}$. So any velocity vector field $v\in V_{N}$ is written as
$$v(x) = (v^{(1)}(x),v^{(2)}(x))  = \left(\sum_{i,j = 1}^N K(x,(x_i,x_j)) m^{(1)}_{i,j},\sum_{k,l = 1}^N K(x,(x_k,x_l)) m^{(2)}_{k,l}\right)$$
where $(m^{(1)},m^{(2)})$ are the momenta matrices which define the momentum linear functional 
$$m = \left(\sum_{i,j=1}^N m^{(1)}_{i,j}\delta_{(x_i,x_j)}, \sum_{k,l=1}^N m^{(2)}_{k,l}\delta_{(x_k,x_l)}\right)\in V_N^*$$
that canonically represents elements of $V_{N}^*$. Note that $\delta\in C^1(\mathbb{R}^d,\mathbb{R})^*$ is the dirac delta functional defined by $(\delta_{x} , \psi) = \psi(x)$ for $\psi\in C^1$). If $m\in V_{N}^*$ then, using the canonical duality operator $\mathbb{K}:V^* \to V,$ we have $\mathbb{K} m = v$. The norm of $v\in V_{N}$ is determined by $m$ and a matrix $\tilde{K}$. 
\begin{equation}\label{momentum}
\|v\|_{V}^2 = \sum_{k=1}^2\left(m^{(k)}|v^{(k)}\right)= \sum_{k=1}^2\sum_{i,j}^N m^{(k)}_{i,j}v^{(k)}((x_i,x_j))=\sum_{k=1}^2\sum_{i,j}^N m^{(k)}_{i,j}\left(\tilde{K} m^{(k)}\right)_{i,j}
\end{equation}
where 
$$[\tilde{K}]_{i,j} = K((x_i,x_j),(0,0)) = \kappa(|x_i-x_j|).$$
Remark as well that $\|m\|_{V^*}^2 = \|v\|^2_V$. Note that $V_N$ does not have a mesh parameter and is called mesh free, but the grid we chose is called its collocation points.

\subsubsection{Evolution Equation}
For a given $u\in \mathcal{U}_{T, N}$, $m\in \mathcal{M}_{T, N}$, and initial data function $F_0\in L^{r^*}$ consider  the time discrete, spatially ``continuous'' scheme 
\begin{align*}
      F(k+1) &= g^{-1}\left(g(f(k)) + \Delta t g'(f(k)) V(k) \right),\\
    V(k)& =\chi(|\nabla f(k)|) u_k -  \nabla f(k)^T (\mathbb{K}m(k)) + I(f(k))\\
    f(k) &= \eta_{\tau}\star F(k) .
\end{align*} 
This scheme defines a  $(F(k))_{k=1}^T\in Q_{T}$. We discretize this scheme spatially by updating  $N$ by $N$ matrices rather than functions:
\begin{align*}
    f(k+1) & = M\star F(k)\\
 F(k) &= g^{-1}\left(g(f(k)) + \Delta t g'(f(k)) \cdot V(k)\right)\\
    V(k)&= u(k)\cdot \left( (DM \star f(k) )\cdot (DM \star f(k) ) + (DM^T \star f(k) )\cdot(DM^T \star f(k) ) + \Delta t \, \epsilon\, {\rm Id}\right)^{1/2} \\
    &-  v^{(1)}(k) \cdot (DM \star f(k)) -v^{(2)}(k)^k\cdot (DM^T \star f(k) )+ I(f(k))\\
    v^{(n)}(k) &= 
    \tilde{K} m^{(n)}(k), \quad n=1,2
\end{align*}
where we use $A\cdot B$ to denote entry-wise multiplication of matrices $A$ and $B$ with the same shape and take $\epsilon = 10^{-16}$. This numerical scheme is equivalent to mapping the function $F(k+1)$ to $L^{r}_{N}$ through $[F_{ij}(k)=\sum_{i,j=1}^N F(k, (x_i,x_j))\mathbf{1}_{V_{ij}(k)}(x_i,x_j)$ at each time step before computing $f(k+1)$. Remark that we avoid finite difference derivatives and directly compute the derivative of $\eta_{\tau}$ at our grid points within the definitions of $M,DM.$  Within our implementation, we define $g$ as follows. For small $a\geq 0$ and small $\mu >0,$ let $g:[-a,1+a]\to \mathbb{R}$ be given by $g(x) = \dfrac{1}{2} + \mu \log\left(\dfrac{x+a}{1+a-x}\right).$ 

\subsubsection{Objective Function}
 In what follows we view $f_T$ as a function of $f_0\in L^{r^*}$, $u\in \mathcal{U}_{T,N},m\in \mathcal{M}_{T, N}$ and write $f_T(f_0,u,m)$. We define $E:\mathcal{U}_{T,N}\times \mathcal{M}_{T, N}\to \mathbb{R}$  by
$$E(u,v) = C_{\mathrm{top}} \|u\|^p_{\mathcal{U}} + \|v(m)\|^p_{\mathcal{V}} + U(f_T(f_0, u,m)),$$
where we use the formulas \ref{unorm} and \ref{momentum} to compute the first two terms and $U:W^{1,r^*}\to \mathbb{R}$ is a $C^1$ functional that is bounded below by zero. For example, one may take $U(f_T(f_0,u,m))$ as the discretization of $C\|f_T(f_0,u,v) - f_{T}(f_{{\rm target}, v,m}) \|_{W^{1,r^*}}^{r^*}$ where $C>0$ is a large constant. 

\subsubsection{Gradient of E}
To perform gradient descent, one can use a Lagrange multiplier method
\begin{equation}
    \mathcal{L}(F,u,m,\lambda) = C_{\mathrm{top}} \|u\|_{\mathcal{U}}^p + \|v(m)\|_{\mathcal{V}}^p + U(\eta_{\sigma \sqrt{t}}\star F_{T}) + \sum_{k=1}^{T-1}\left(\lambda_{k+1} \mid F_{k+1} -\tilde{G}(F_k,u^k,m^k)\right) 
\end{equation}
to introduce a discrete costate $\lambda:\{1,...,T\}\to M^{N\times N}$ whose evolution is given by
\begin{align}
    [\lambda(T)]_{ij} &= - [\partial_{F(T)} U(f(T))]_{ij}\\
    \lambda(k) &=( \partial_{F(k)} G(k))^* \lambda(k+1),\,k=1,...,T-1\\
    G(k) &=  g^{-1}(g(M \star F(k)) + g'(M \star F(k))V(k))
\end{align}
The derivatives of our objective function $E$ with respect to $u$ and $m$ are given by
$$[\partial_u E(u,m)(k)]_{ij} = p\Delta t C_{\mathrm{top}} \left((\Delta x)^2 \sum_{l,l'=1}^N (u_{ll'}(k))^r\right)^{(p-r)/r} |u_{ij}(k)|^{r-2} u_{ij}(k) - [(\partial_{u}G(k))^* \lambda(k)]_{ij}$$
and for each $n=1,2$
$$[\partial_{m^{(n)}} E(u,m)(k)]{ij} = p\Delta t \left(\sum_{l,l'=1}^N m^{(n)}_{ll'}(k) v^{(n)}_{ll'}(k)\right)^{(p-2)/2}(v^{(n)})_{ij}(k) - [(\partial_{m^{(n)}}G(k))^* \lambda(k)]_{ij}$$
where
$(v^{(n)})_{ij}(k) = \left(
\tilde{K} m^{(n)}(k)\right)_{i,j}.$

\subsection{Experimental results}
By using a LBFGS method with the definition of $E(u,v)$ and $dE(u,v)$, we performed gradient descent with initial guess $u,v=0$ using MATLAB.
We tested several initial and final images, with different or identical topology.
Unless mentioned otherwise, our experiments use $T=30$, $L=1$, $N= 150$, $\sigma = 1/10$ and $C_{\mathrm{top}} = 10^8$. These are referred to, below, as default parameters.

\subsubsection{Sets of discs}
We start with a simple example, in \cref{fig:1Cto2C}, where an original disc shape is transformed into a target that contains two discs. The result of the optimal evolution  is provided with a progressive elongation of the disc, creating a neck that breaks to provide two discs. \Cref{fig:1Cto2Ctop} provides an illustration of the relative impact of the topological and advection controls ($u$ and $v$) when the parameter is changed to $10^6$, $10^8$ (same as in \cref{fig:1Cto2C}) and $10^{10}$. The diffeomorphic transformation becomes larger, but never changes (as expected) the topology of the original shape.

The images in \cref{fig:1Cto2Ctop} must be interpreted with some care. Displays in the second column provide the solution at time $t=1$ of the equation $\partial_t \varphi = v \circ \varphi$. When the effect of the control $u$ on the evolution of $f$ is not negligible, the original image transformed by $\varphi$ is generally behind the one controlled by $u$ and $v$ together. This is illustrated by the red dots stopping short of the target shapes. The transformation of background points (in blue)  completes the description of the  function $\varphi$ and of how it operates on the actual moving shapes. Similarly, the last column in the figure illustrates the overall action of the normal field $u$ had the diffeomorphic transformation been absent. It provides qualitative insights on this part of the process, with, in this case, a combination of horizontal stretching and vertical pinching.

\begin{figure}\label{Top1}
    \centering
    \includegraphics[width=1\linewidth]{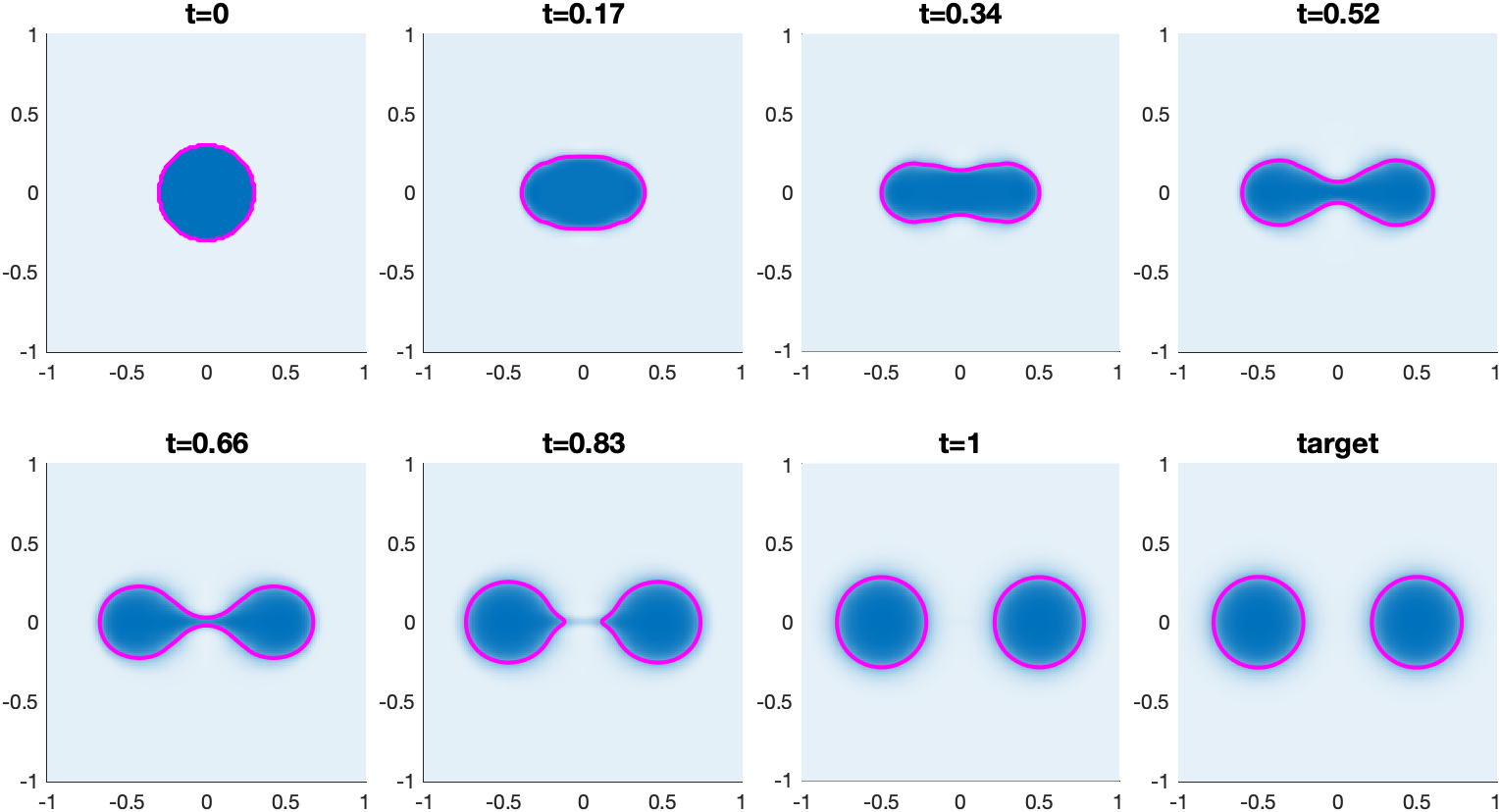}
    \caption{Here our initial shape was the unit disk and the final shape was two unit disks under a Gaussian blur.  Presented is the trajectory of the function f (superimposed with its main contour in magenta) at seven time points  and the last image is the target image. (Using default parameters.)
    \label{fig:1Cto2C}}
\end{figure}

\begin{figure}
\centering
\includegraphics[width=0.8\linewidth]{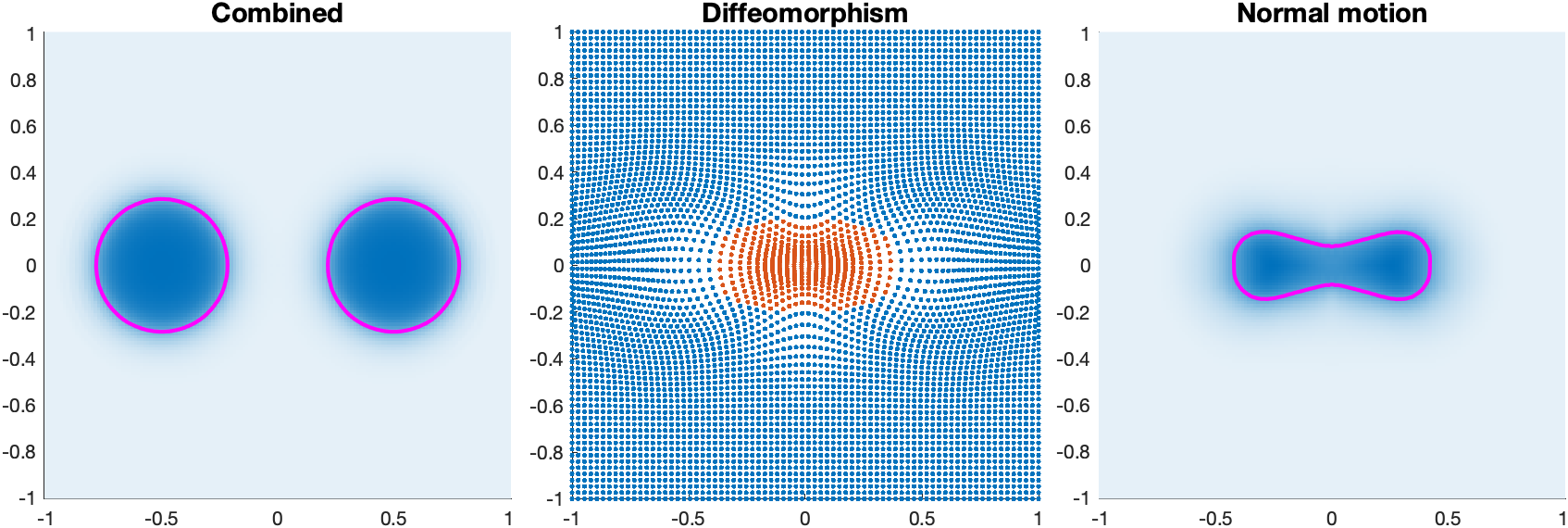}\\
\includegraphics[width=0.8\linewidth]{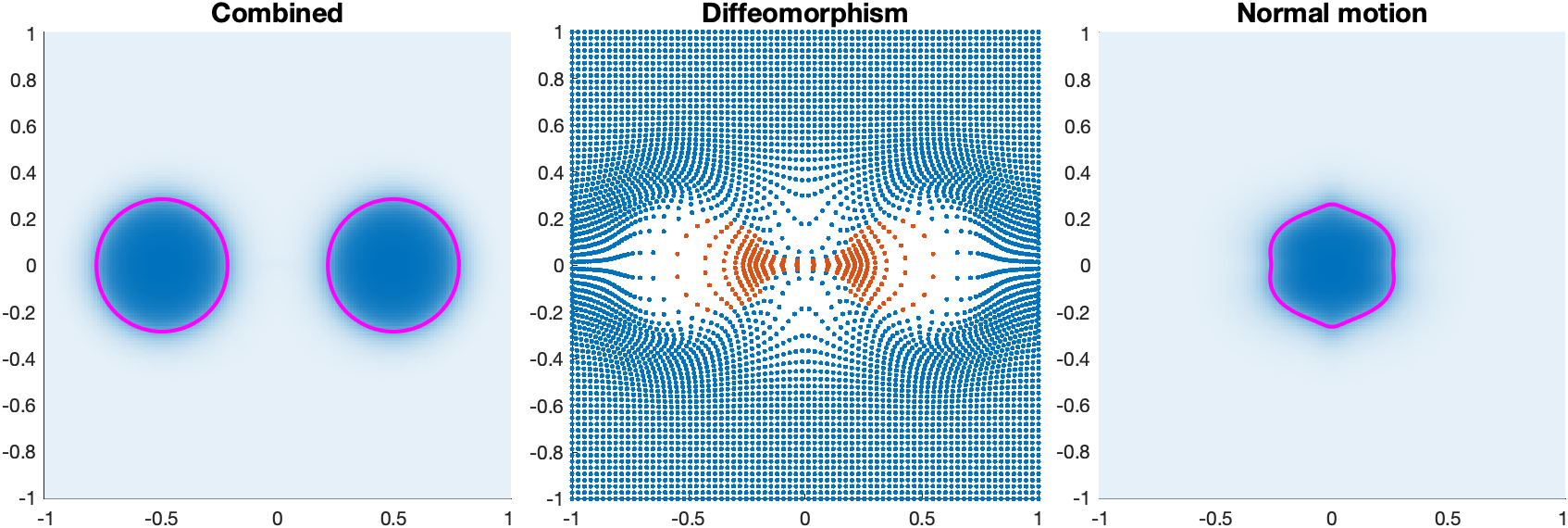}\\
\includegraphics[width=0.8\linewidth]{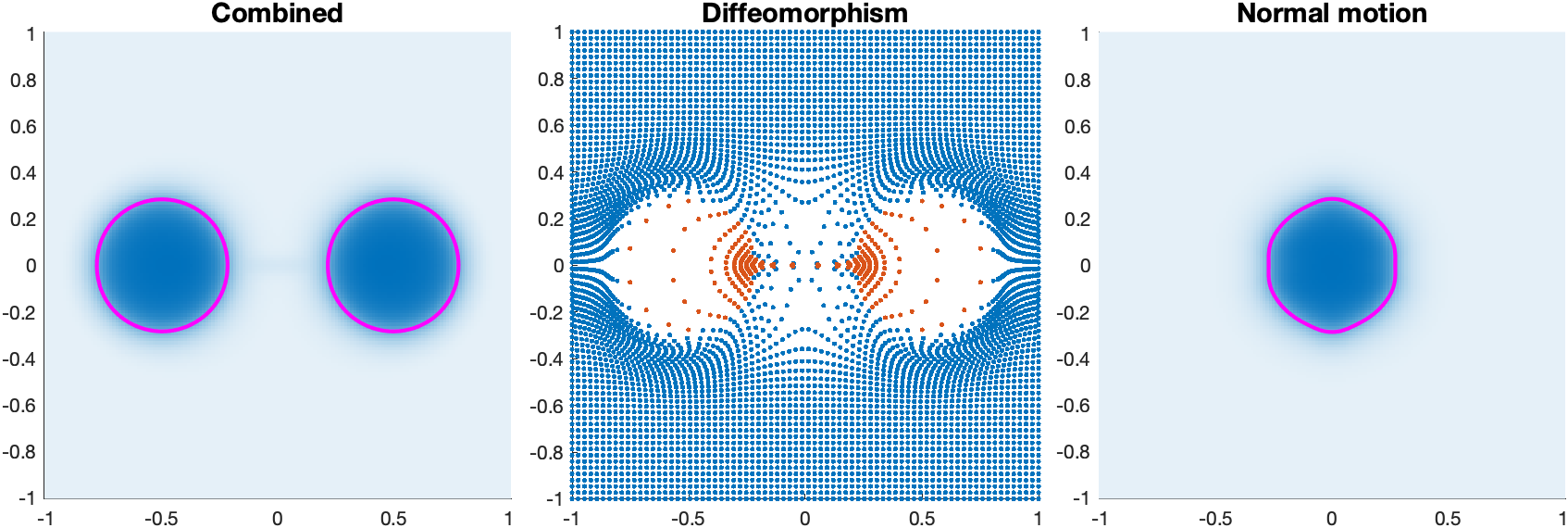}
\caption{Illustration of the impact of the penalty parameter. The rows correspond, from top to bottom, to $C_{\mathrm{top}} = 10^6$, $10^8$ and $10^{10}$, using default parameters otherwise. Each row provides, from left to right, the final image ($t=1$) resulting from the state equation, the diffeomorphism associated to the smooth vector field $v$ (solution to the flow equation at $t=1$), visualized by applying if to a regular grid, with different colors for connected components of the original image,  and the result of only applying the normal motion (setting $v=0$) in the state equation. 
\label{fig:1Cto2Ctop}
}
\end{figure}

Our second example, in \cref{fig:3Cto4C,fig:4Cto3C}, transforms three discs into four and conversely. It first illustrates that our process is not symmetric, as the way two components merges is quite different from the way one component splits. One can nonetheless notice similarities between the deforming shapes at time $t$ for the forward process and $1-t$ for the backward one.
\begin{figure}
    \centering
    \includegraphics[width=1\linewidth]{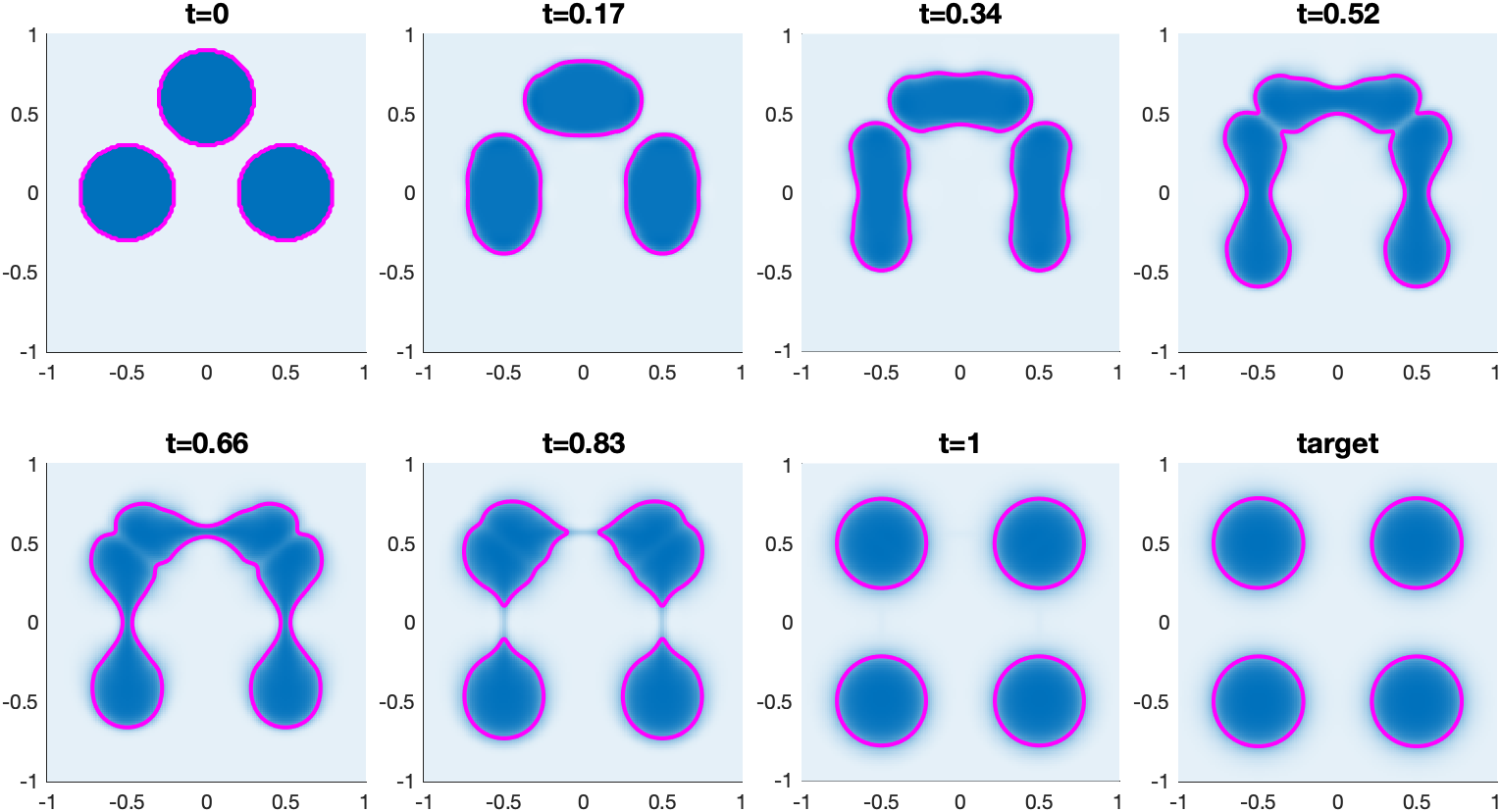}
    \caption{Here our initial shape was three unit disks and the final shape was four unit disks under a Gaussian blur. (Using default parameters.) 
    \label{fig:3Cto4C}}
\end{figure}
\begin{figure}
    \centering
    \includegraphics[width=1\linewidth]{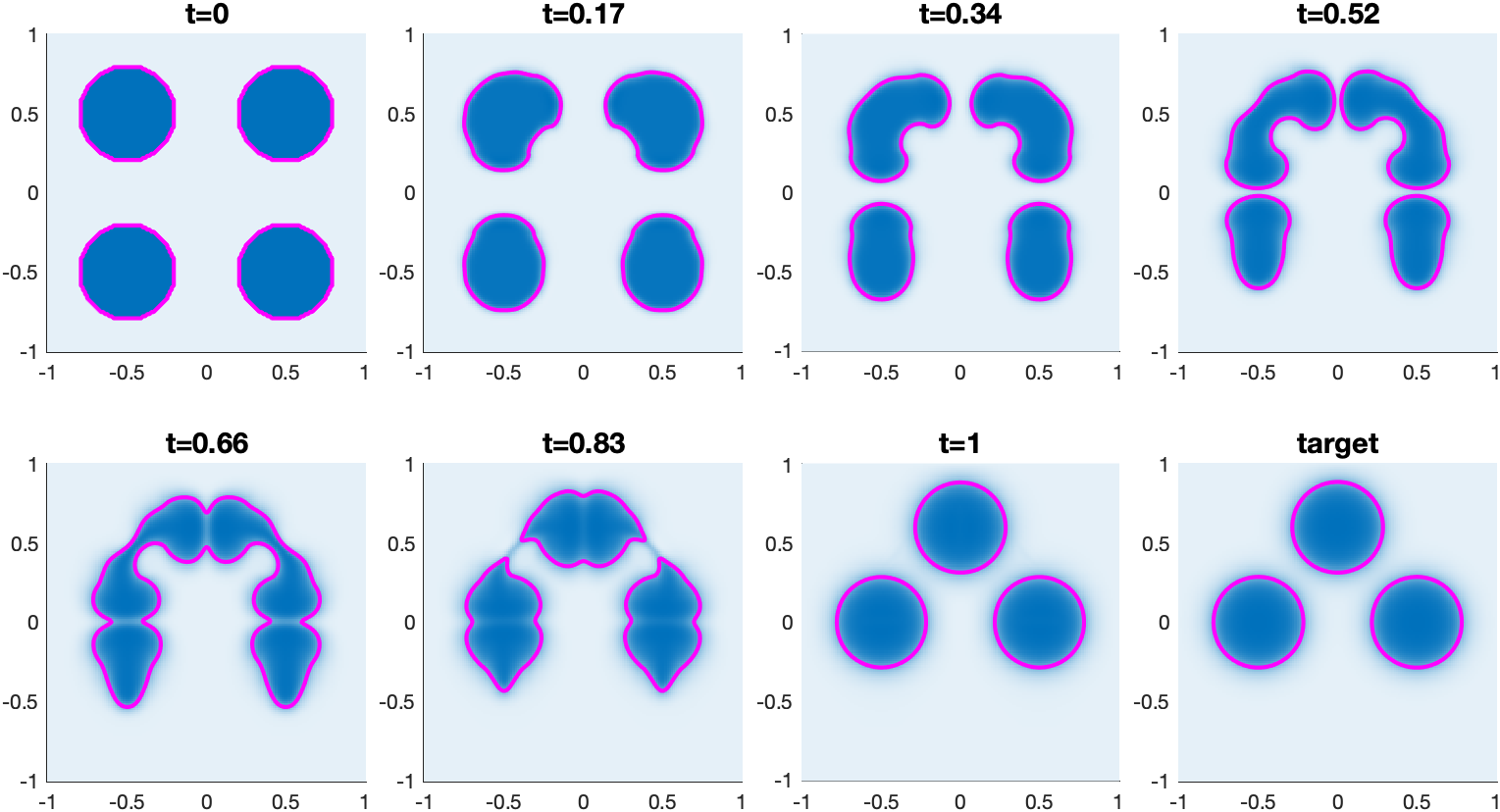}
    \caption{Reverse comparison from that of \cref{fig:3Cto4C}. (Using default parameters.) 
    \label{fig:4Cto3C}}
\end{figure}

\subsubsection{Complex Changes of Topology and Geometry}
The examples in this section depict shapes based on had gestures. The first example, in \cref{fig:okaytoopen}, transformed a hand performing an ``okay'' sign into an open hand. \cref{fig:okaytoopen2} illustrate the respective effects of diffeomorphic and topological changes, where, in this case, most of the motion is operated by the advection, while the normal motion mainly breaks the bridge between the thumb and the index finger.

\begin{figure}
    \centering
    \includegraphics[width=1\linewidth]{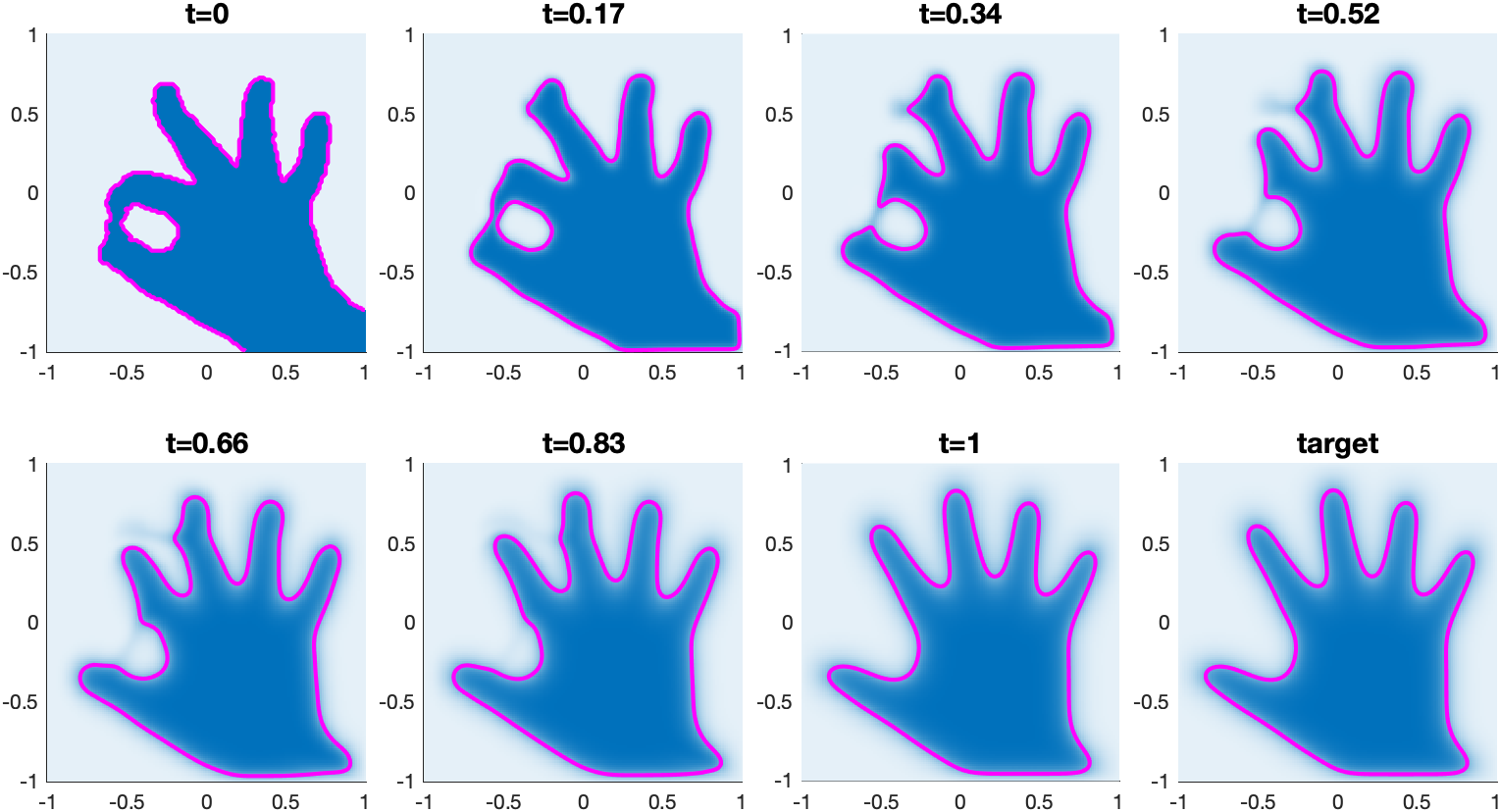}
    \caption{Here our initial shape was an hand giving the 'okay' gesture and the final shape was an open hand. (Using default parameters.)
    \label{fig:okaytoopen}}
\end{figure}

\begin{figure}
    \centering
    \includegraphics[width=1\linewidth]{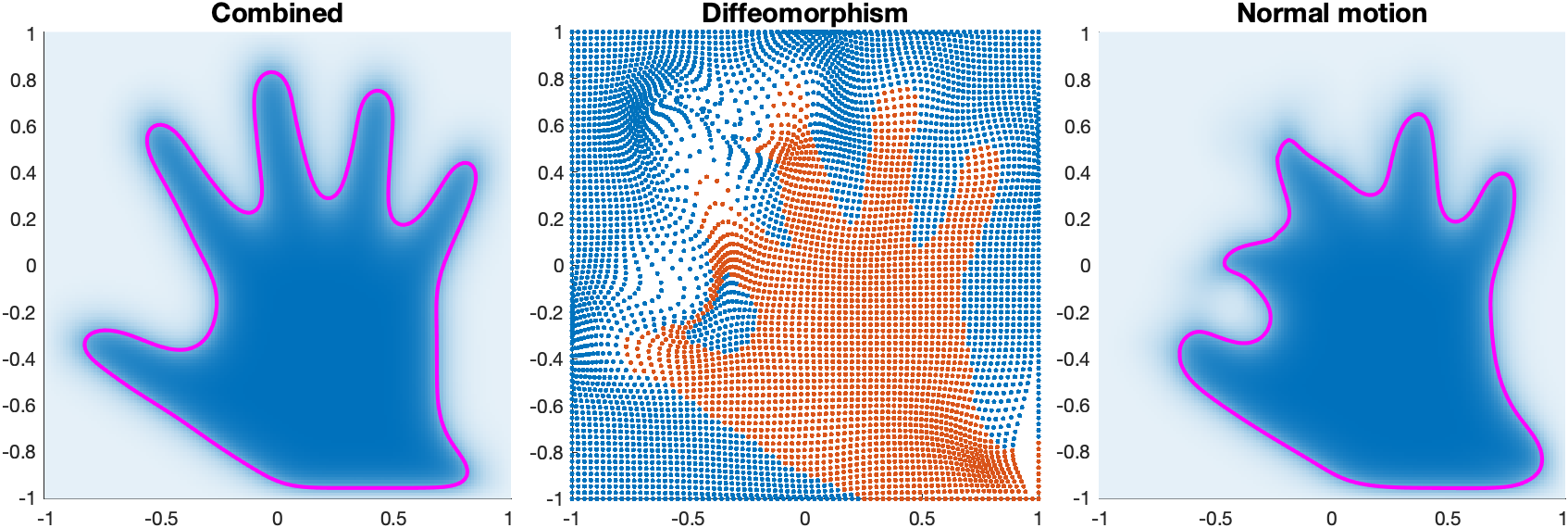}
    \caption{Visualization of the diffeomorphic and topological changes in the example of \cref{fig:okaytoopen}. 
    \label{fig:okaytoopen2}}
\end{figure}

Figures \ref{fig:spockhands1}, \ref{fig:spockhands2} and \ref{fig:spockhands3} compare a hand performing a ``vulcan salute'' in which the hand is parted between the middle and ring finger and an  open hand. Even though the two shapes are topologically equivalent, the trajectories followed by the transformations exhibit significant differences according to the value of the coefficient $C_{\mathrm{top}}$. For $C_{\mathrm{top}} = 10^6$ (\cref{fig:spockhands1}), an indent is created on the middle finger, leading to a hole positioning itself at the separation between the first two fingers. A similar indent is created in \cref{fig:spockhands2}, with $C_{\mathrm{top}} = 10^6$,  but does not lead to a topological change. The indent almost completely disappears with $C_{\mathrm{top}} = 10^{10}$ (\cref{fig:spockhands3}). 

Our last illustration compares a closed hand with a hand with two raised fingers (\cref{fig:rockscissors}).

\begin{figure}
    \centering
    \includegraphics[width=1\linewidth]{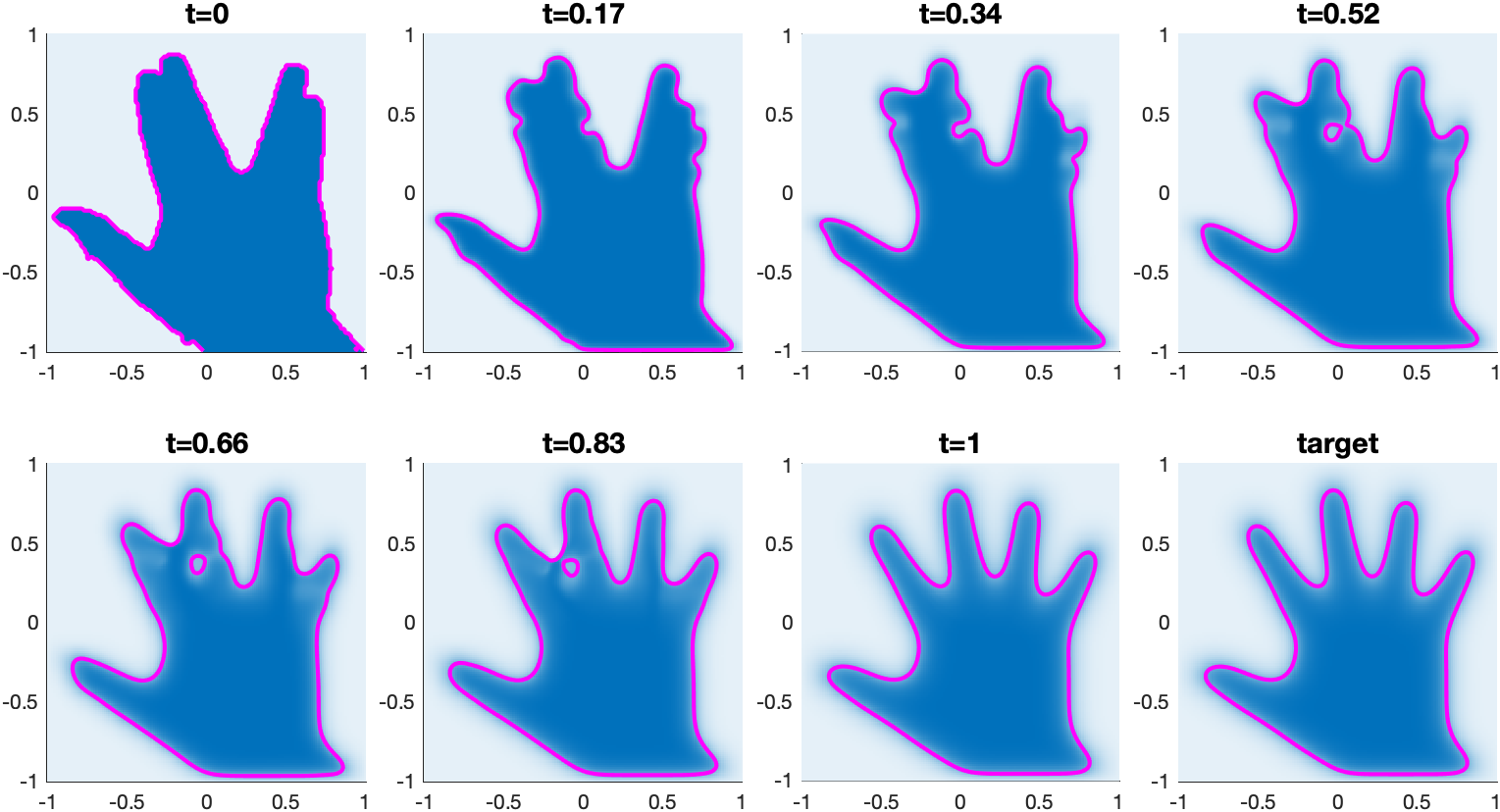}
    \caption{Here our initial shape is a hand with two pairs of fingers attached together and the last shape is an open hand, with $C_{\mathrm{top}} = 10^6$ and default parameters otherwise. 
    \label{fig:spockhands1}}
\end{figure}

\begin{figure}
    \centering
    \includegraphics[width=1\linewidth]{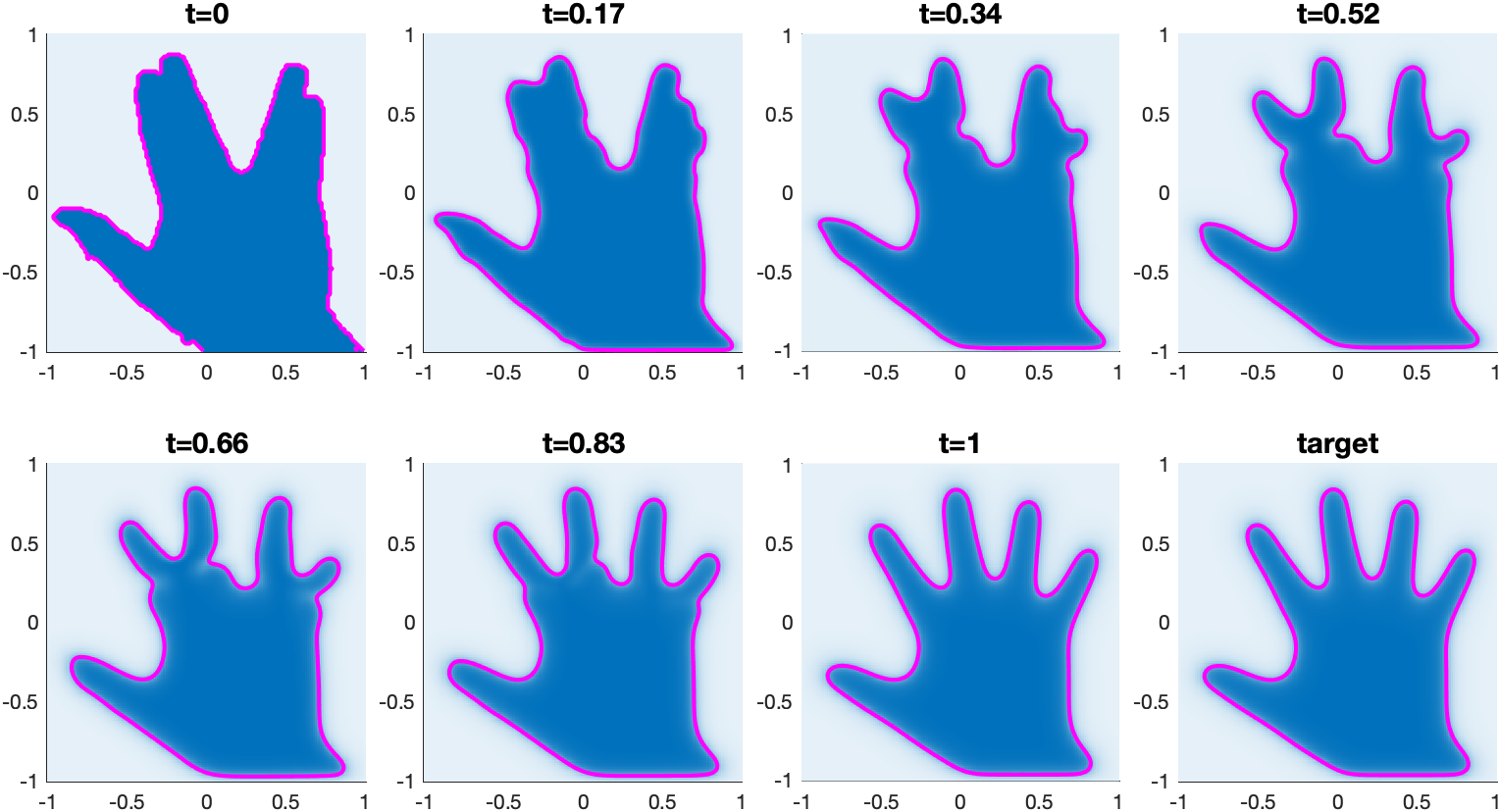}
    \caption{Same as \cref{fig:spockhands1}, with $C_{\mathrm{top}} = 10^8$. 
    \label{fig:spockhands2}}
\end{figure}

\begin{figure}
    \centering
    \includegraphics[width=1\linewidth]{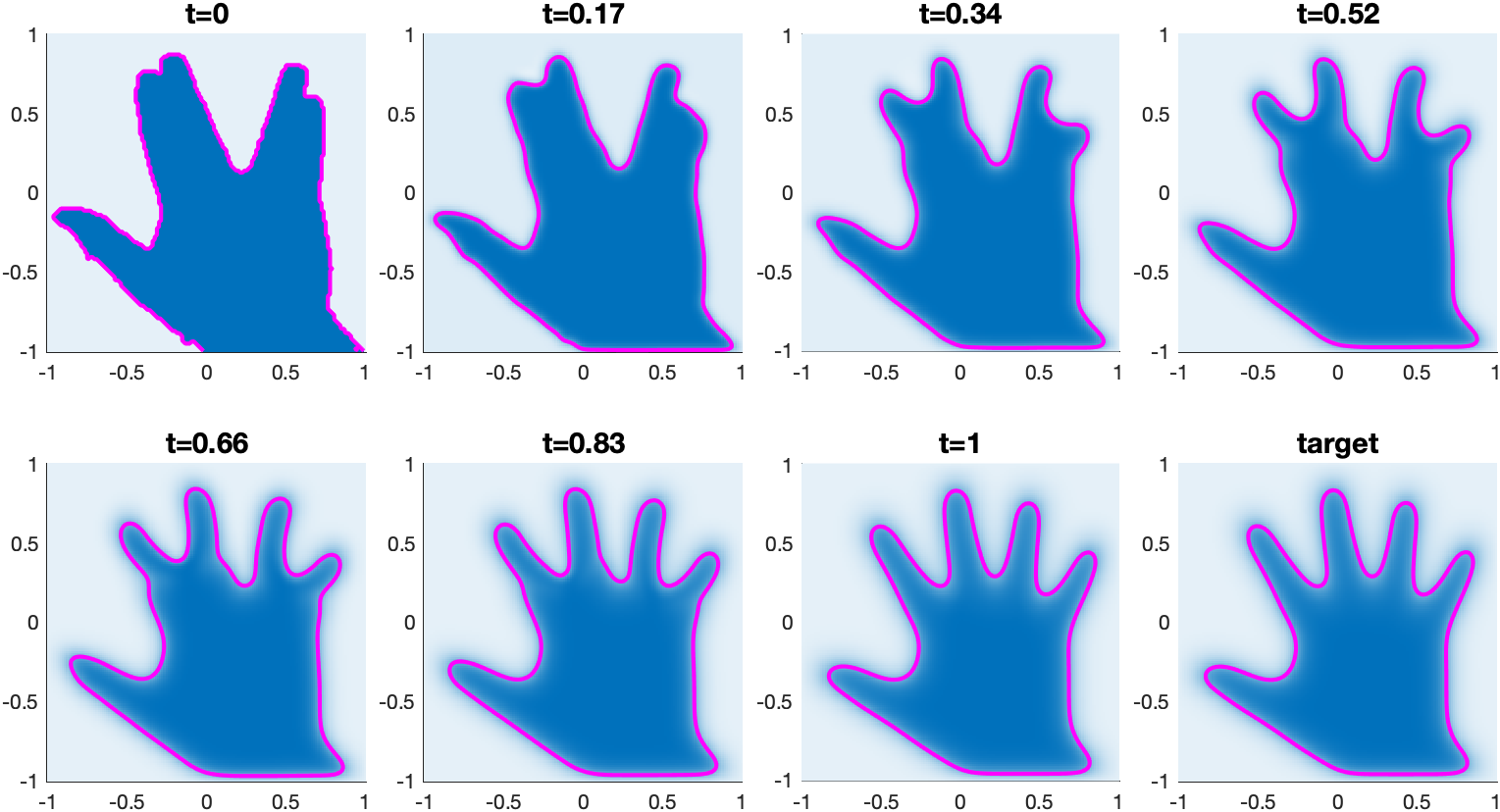}
    \caption{ Same as \cref{fig:spockhands1}, with $C_{\mathrm{top}} = 10^{10}$. 
    \label{fig:spockhands3}}
\end{figure}

\begin{figure}
    \centering
    \includegraphics[width=1\linewidth]{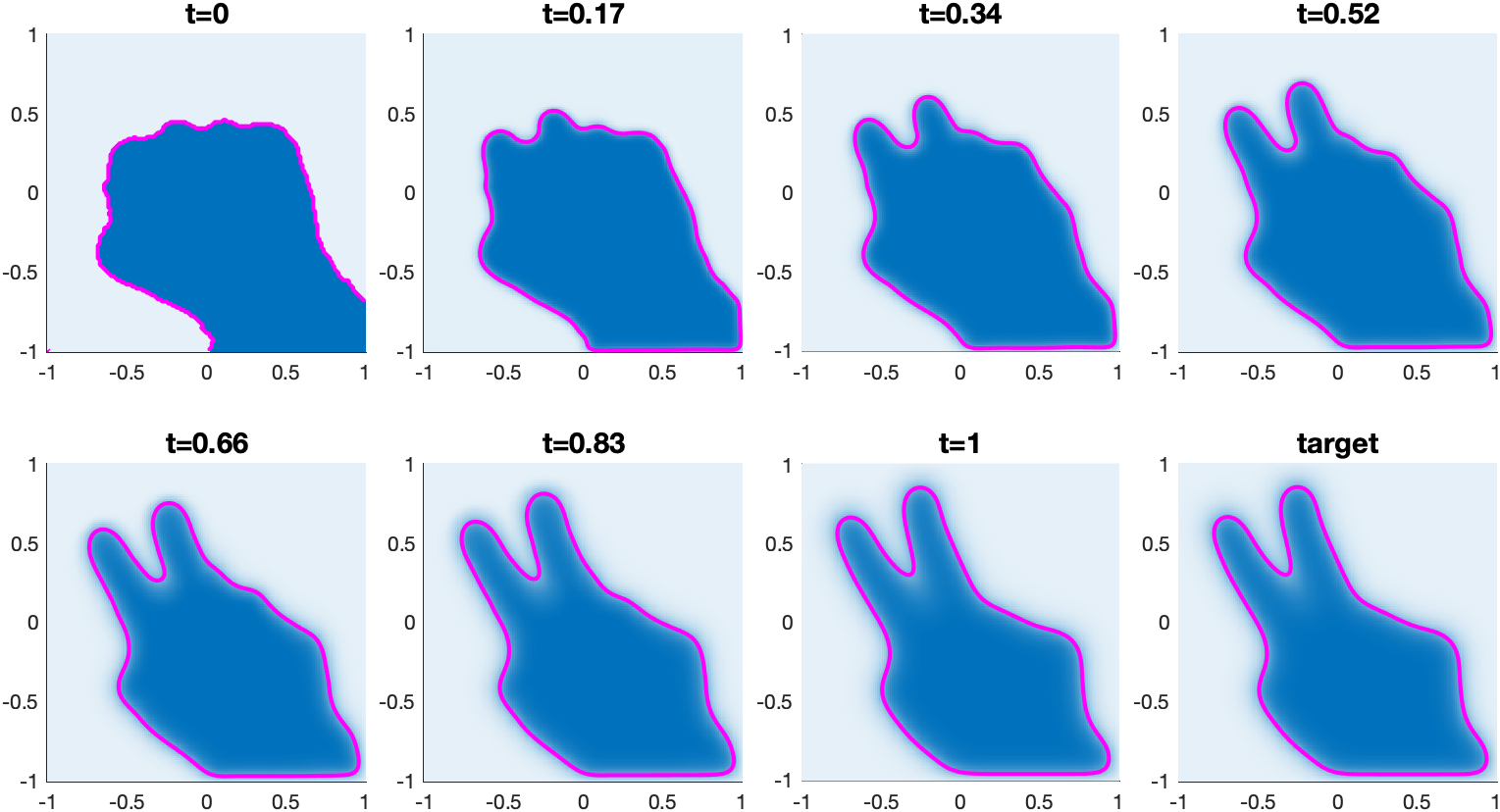}
    \caption{Here our initial shape is a hand with a closed fist and the target shape is a hand with two fingers up, the 'peace' gesture. (Using default parameters.) 
    \label{fig:rockscissors}}
\end{figure}


\section{Summary}
In conclusion, this is a new alignment method that can align any shape’s characteristic function to a target shape’s characteristic function which has undergone mean curvature flow. We prove the existence of solutions given any choice of controls $(u,v)\in \mathcal{U}\times \mathcal{V}$ , set up a discrete time soft end point minimization problem. With a Lagrange Multiplier result we characterized minimizers for the discrete time problem, and demonstrate simulations of this approach. Our numerical simulations also enforce the maximum bounded principle. With $(u,v)=(0,0)$ our numerical simulations demonstrate approximate mean curvature flow. And one can analyze the role of geometric change in shape transformation and topological change by weighing the norms of $u$ and $v$. There are many open questions concerning this type of alignment. In particular what remains is proving the continuity of mild solutions with respect to the controls, the existence and uniqueness of strong solutions, and the existence of minimizers to the soft end point problem. 

\section*{Acknowledgements}
This research was partially funded at the Division of Applied Mathematics at Brown University by the NSF Research Training Grant: Mathematics of Information and Data with Applications to Science, the Brown Graduate School Presidential Fellowship, and Johns Hopkins University while D. Solano worked as an Assistant Research Engineer.

\bibliography{Bibliography}


\end{document}